\documentclass[preprint,12pt]{elsarticle}

\usepackage{amssymb}
\usepackage{amsmath}
\usepackage{amsthm}
\usepackage{hyperref}
\usepackage{xcolor}
\newtheorem{theorem}{Theorem}[section]
\newtheorem{lemma}[theorem]{Lemma}
\newtheorem{corollary}[theorem]{Corollary}
\newtheorem{proposition}[theorem]{Proposition}

\newcommand{\reals}{\mathbb{R}}
\newcommand{\complex}{\mathbb{C}}
\newcommand{\field}{\mathbb F}
\newcommand{\img}{\textup{Im}}
\newcommand{\chara}{\textup{char}}
\newcommand{\diag}{\textup{diag}}
\newcommand{\tr}{\textup{tr}}

\journal{Linear Algebra and its Applications}

\begin{document}

\begin{frontmatter}



\title{On commutators of unipotent matrices of index 2} %% Article title

\author[sd]{Kennett L. Dela Rosa}
\author[sd]{Juan Paolo C. Santos} %% Author name

\affiliation[sd]{organization={Institute of Mathematics, University of the Philippines Diliman},%Department and Organization
           addressline={C.P. Garcia Ave., UP Campus, Diliman}, 
            city={Quezon City},
            postcode={1101},
            country={Philippines}}

%% Abstract
\begin{abstract}
 A commutator of unipotent matrices of index 2 is a matrix of the form $XYX^{-1}Y^{-1}$, where $X$ and $Y$ are unipotent matrices of index 2, that is, $X\ne I_n$, $Y\ne I_n$, and $(X-I_n)^2=(Y-I_n)^2=0_n$. If $n>2$ and $\mathbb F$ is a field with $|\field|\geq 4$, then it is shown that every $n\times n$ matrix over $\field$ with determinant 1 is a product of at most four commutators of unipotent matrices of index 2. Consequently, every $n\times n$ matrix over $\field$ with determinant 1 is a product of at most eight unipotent matrices of index 2. Conditions on $\field$ are given that improve the upper bound on the commutator factors from four to three or two. The situation for $n=2$ is also considered. This study reveals a connection between factorability into commutators of unipotent matrices and properties of $\field$ such as its characteristic or its set of perfect squares.
\end{abstract}



%% Keywords
%\MSC[2020] 15A21 \sep 15A23 \sep 15A30 \sep 16S50 \sep 16U99 \sep 20H20
\end{frontmatter}




%%%%%%%%%%%%%%%%%%%%%%%%%%%%%%%%%%%%%%%%%%%%%%%%%%%%%%%%%%%%%%%%%%%%%%
%%%%%%%%%%%%%%%%%%%%%%%%%%%%%%%%%%%%%%%%%%%%%%%%%%%%%%%%%%%%%%%%%%%%%%
%%%%%%%%%%%%%%%%%%%%%%%%%INTRODUCTION$$%%%%%%%%%%%%%%%%%%%%%%%%%%%%%%%
%%%%%%%%%%%%%%%%%%%%%%%%%%%%%%%%%%%%%%%%%%%%%%%%%%%%%%%%%%%%%%%%%%%%%%
%%%%%%%%%%%%%%%%%%%%%%%%%%%%%%%%%%%%%%%%%%%%%%%%%%%%%%%%%%%%%%%%%%%%%%
\section{Introduction}
\label{intro}

Throughout this paper, $\mathbb F$ denotes a field and $\textup{char}(\mathbb F)$ its characteristic. Let $M_n(\mathbb F)$ be the set of all $n\times n$ matrices with entries in $\mathbb F$, $GL_n(\field)$ be the group of all $n\times n$ matrices with nonzero determinant, and $SL_n(\mathbb F)$ be the group of all $n\times n$ matrices with determinant 1. The $n\times n$ identity matrix is denoted by  $I_n$ while the $n\times n$ zero matrix is denoted by $0_n$. If $X,Y\in GL_n(\mathbb F)$, set $[X,Y]:=XYX^{-1}Y^{-1}.$

An $A\in M_n(\mathbb F)$ is a (\textit{multiplicative\textup{)} commutator} if $A=[X,Y]$ for some $X,Y\in GL_n(\mathbb F)$. Since the determinant is multiplicative, it follows that
\begin{center}
$\det(A)=\det(XYX^{-1}Y^{-1})=\det(X)\det(Y)\det(X)^{-1}\det(Y)^{-1}=1$.    
\end{center} %An $A\in M_n(\mathbb F)$ is an \textit{involution} if $A^2=I_n$. 
Hence, products of commutators are in $SL_n(\mathbb F)$. In fact, more is true: Shoda showed in \cite{shoda} that for an algebraically closed $\mathbb F$, every matrix in $SL_n(\mathbb F)$ is a commutator while Thompson in \cite{thompson} extended this result for all fields except for the case when $n=2$ and $|\field|=2$. Hence, additional conditions are usually imposed on $X$ and $Y$, and it is natural to ask in this case if the commutators generate certain matrix groups. For example, commutators of $J$-symmetries were considered in \cite{dd}. In \cite{zheng}, it was shown that every matrix in $SL_n(\mathbb R)$ or $SL_n(\mathbb C)$ is a product of at most two commutators of involutions. The result was extended to arbitrary fields with $\chara(\field)\ne2$ in \cite{hou2} and then in \cite{son} where the characteristic condition was removed but the assumption is that $|\field|\geq 3$.

%it was shown that $S$ is a product of at most two commutators of involutions provided that , while some observations have been made on the factorization of matrices in $SL_n(\mathbb F)$ when $|\mathbb F|=2$. 

An $A\in M_n(\mathbb F)$ is said to be a \textit{unipotent matrix of index $k$} or \textit{${\cal U}_k$-matrix} if \begin{center}$(A-I_n)^k=0_n$ and $(A-I_n)^{k-1}\ne0_n$.\end{center} In \cite{hou}, Hou proved that every matrix in $SL_n(\complex)$ is a product of at most two commutators of ${\cal U}_2$-matrices. Since a commutator of ${\cal U}_2$-matrices is a product of two ${\cal U}_2$-matrices (see Proposition \ref{rmk:houtowang} for a generalization), Hou's result \cite[Theorem 1.1]{hou} implies Wang and Wu's result \cite[Theorem 3.5]{wangwu}, which asserts that every matrix in $SL_n(\mathbb C) $ is a product of at most four ${\cal U}_2$-matrices. Recently, Ha and Toan in \cite{ha} extended Hou's result to matrices in $GL_n(\mathbb H)$ where $\mathbb H$ is the real quaternion division ring. The authors showed that every matrix in $SL_n(\mathbb H)$ is a product of at most three commutators of ${\cal U}_2$-matrices. Products of ${\cal U}_2$-matrices in $GL_n(D)$ was investigated in \cite{bien} where $D$ is a division ring.

%\indent In this paper, we consider the decomposition of matrices into commutators of unipotent matrices of index 2 over a given field $\field$. 

\indent The goal of the paper is to generalize the main results in \cite{hou} and \cite{wangwu} for arbitrary fields $\mathbb F$. In particular, we prove that every matrix in $SL_n(\field)$ is a product of at most four commutators of ${\cal U}_2$-matrices provided that  $n>2$ and $|\field|\geq 4$ (Theorem \ref{mainthm}). For $SL_2(\field)$, the upper bound can be improved: if $|\mathbb F|\geq 4$, then the upper bound is \textit{three} or \textit{two} depending on $\mathbb F$ (Theorem \ref{cor:SL2}) but if $|\mathbb F|\leq 3$, then only those in the commutator subgroup $SL_2(\mathbb F)'$ can be considered and in this case, $|\mathbb{F}|-1$ is the upper bound (Theorem \ref{thm:Z2Z3}). We also provide conditions on $\field$ that improve the upper bound to \textit{three} or \textit{two} (Theorems \ref{mainthm:char2}-\ref{mainthm6}). 

%As an immediate consequence of Theorems \ref{cor:SL2} and \ref{mainthm}, we obtain an analogue of Wang and Wu's result in \cite{wangwu} (Corollary \ref{cor:wangwu}). Theorem \ref{mainthm6} is a generalization of \cite[Theorem 1.1]{hou} and \cite[Theorem 3.5]{wangwu}.

%For $SL_2(\field)$, we find the least upper bound over matrices in $SL_2(\mathbb F)$ of the number of commutator factors in the commutator subgroup $SL_2(\mathbb F)'$. If $|\mathbb F|\geq 4$, then \textit{three} is the least upper bound for some $\mathbb F$ (Theorem \ref{cor:SL2}) but if $|\mathbb F|\leq 3$, then only those in $SL_2(\mathbb F)'$ can be considered and in this case, \textit{two} is the least upper bound for some $\mathbb F$ (Theorem \ref{thm:Z2Z3}). We also provide different conditions on $\field$ that reduce the number of commutator of ${\cal U}_2$-matrix factors from four to \textit{two} or \textit{three} (Proposition \ref{mainthm6}-\ref{mainthm:char2}). As a consequence of Theorem \ref{mainthm}, we obtain an analogue of Wang and Wu's result in \cite{wangwu}, namely, every matrix in $SL_n(\field)$ is a product of at most eight ${\cal U}_2$-matrices provided that $|\field|\notin\{2,3,4\}$ and $n>2$ (Corollary \ref{cor:wangwu}).

The paper is organized as follows. In Section \ref{unipmat}, we prove properties of unipotent matrices and their commutators. In Section \ref{sec:SL2}, we first make observations for $SL_2(\mathbb F)$ and in Section \ref{sec:SLn}, we extend some results to the general case $SL_n(\field)$. In Section \ref{sec:reduc}, we identify conditions on $\field$ that improve the upper bound on the commutator factors in Theorem \ref{mainthm}. 

%%%%%%%%%%%%%%%%%%%%%%%%%%%%%%%%%%%%%%%%%%%%%%%%%%%%%%%%%%%%%%%%%%%%%%
%%%%%%%%%%%%%%%%%%%%%%%%%%%%%%%%%%%%%%%%%%%%%%%%%%%%%%%%%%%%%%%%%%%%%%
%%%%%%%%%%%%%%%%%%%%%%UNIPOTENT MATRICES%%%%%%%%%%%%%%%%%%%%%%%%%%%%%%
%%%%%%%%%%%%%%%%%%%%%%%%%%%%%%%%%%%%%%%%%%%%%%%%%%%%%%%%%%%%%%%%%%%%%%
%%%%%%%%%%%%%%%%%%%%%%%%%%%%%%%%%%%%%%%%%%%%%%%%%%%%%%%%%%%%%%%%%%%%%%
\section{Unipotent matrices, their commutators, and other preliminaries}
\label{unipmat}

%In this section, we discuss properties of unipotent matrices and their commutators. 

%Let $\field$ be a field, $A\in M_m(\field)$ and $B\in M_n(\field)$. We denote by $A^\top$, $\tr(A)$, and $\det(A)$ the transpose, trace, and determinant of $A$, respectively. Also, we denote by $\sigma(A)$, $m_A(x)$, and $p_A(x)$ the set of all eigenvalues, the minimal polynomial, and the characteristic polynomial of $A$, respectively. The direct sum of $A$ and $B$ is defined as the matrix
%$A\oplus B=\begin{bmatrix}
%A & 0_{m,n} \\
%0_{n,m} & B 
%\end{bmatrix}$. For $k\geq2$, the direct sum of the matrices $A_1,...,A_k$ is defined as %$$\bigoplus\limits_{i=1}^kA_i=A_1\oplus\cdots\oplus A_k=\begin{bmatrix}
 %   A_1 & &\\
 %    & \ddots &\\
 %    & & A_k
%\end{bmatrix}.$$

%We denote by $\textup{diag}(a_1,...,a_n)$ the $n\times n$ diagonal matrix with $a_i$ in the $(i,j)$ entry if $i=j$ and 0, otherwise. We denote by $J_n(\lambda)$ the $n\times n$ upper triangular Jordan block with the eigenvalue $\lambda$.

%Let $\mathbb F$ be a field. 
Recall that $N\in M_n(\field)$ is said to be a \textit{nilpotent matrix of index $k$} if $N^k=0_n$ and $N^{k-1}\ne0_n$. We only consider $n\geq2$ since $\field$ has no nilpotent elements and $k\geq 2$ since $0_n$ is the only index $1$ nilpotent matrix.

Let $A\in M_n(\mathbb F)$ be a ${\cal U}_k$-matrix. Observe that $A=I_n+N$ for some nilpotent matrix $N$ of index $k$. The minimal polynomial of $A$ is $m_A(x)=(x-1)^k$ which implies that its characteristic polynomial is $p_A(x)=(x-1)^n$ and $k\leq n$. Since the only root of $p_A$ is $1\in\field$ and its multiplicity in $m_A$ is $k$, it follows that the Jordan canonical form of $A$ is composed of Jordan blocks $J_{s}(1)$ corresponding to $1$ and the largest Jordan block is $J_k(1)$ (in general, $J_s(\lambda)$ denotes the upper triangular $s\times s$ Jordan block with eigenvalue $\lambda$). These observations are summarized in the following result where $\bigoplus\limits_{s=1}^{\ell } A_s$ denotes the direct sum of matrices $A_1,\ldots,A_\ell$ and $\sigma(A)$ denotes the set of eigenvalues of $A$. 

    \begin{proposition} \label{prop:evalue}
       Let $A\in M_n(\mathbb F)$ be a ${\cal U}_k$-matrix. The Jordan canonical form of $A$ is given by $\bigoplus\limits_{i=1}^\ell J_{n_i}(1)$ for some $\ell\in\mathbb N$ such that $k=n_1\geq\cdots\geq n_\ell>0$. In particular, $k\leq n$, $\sigma(A)=\{1\}$, $\det(A)=1$, and $\tr(A)=n$. 
    \end{proposition}

For $A\in M_n(\field)$ and $P\in GL_n(\field)$, $(PAP^{-1}-I_n)^\ell=P(A-I_n)^\ell P^{-1}$ for any $\ell\in \mathbb N$. If $A\in GL_n(\mathbb F)$, then $A^{-1}(A-I_n)=I_n-A^{-1}=(A-I_n)A^{-1}.$ Hence, $(A^{-1}-I_n)^\ell=(-A^{-1})^\ell (A-I_n)^\ell $ for any $\ell\in \mathbb N$. These observations guarantee the next result.
\begin{proposition}\label{prop:papinv}
    Let $A\in SL_n(\mathbb F)$ and $P\in GL_n(\mathbb F)$. The following are equivalent:
    \begin{enumerate}[(i)]
        \item $A$ is a ${\cal U}_k$-matrix;
        \item $A^{-1}$ is a ${\cal U}_k$-matrix;
        \item $PAP^{-1}$ is a ${\cal U}_k$-matrix.
    \end{enumerate}
\end{proposition}
%Then $A$ is a ${\cal U}_k$-matrix if and only if $PAP^{-1}$ is a ${\cal U}_k$-matrix. Moreover, if $A$ is a ${\cal U}_k$-matrix, then so is $A^{-1}$.

If $n,k\in\mathbb N$ such that $n\geq k\geq 2$, then $I_{n-k}\oplus J_k(1)$ is a ${\cal U}_k$-matrix (only $J_k(1)$ is present when $n=k$). Hence, $I_n=[I_{n-k}\oplus J_k(1),I_{n-k}\oplus J_k(1)]$. Consequently, the set of commutators of ${\cal U}_k$-matrices is non-empty. 

  Let $[X,Y]=XYX^{-1}Y^{-1}$ where $X,Y$ are ${\cal U}_k$-matrices. By Proposition \ref{prop:papinv}, $YX^{-1}Y^{-1}$ is also a ${\cal U}_k$-matrix. The next result is immediate.
\begin{proposition}\label{rmk:houtowang}
        A commutator of ${\cal U}_k$-matrices is a product of at most two ${\cal U}_k$-matrices.
    \end{proposition}

\begin{proposition}\label{id}
The identity matrix is the only scalar matrix that is a commutator of ${\cal U}_k$-matrices.
\end{proposition}
\begin{proof}
Let $\lambda\in\mathbb F$. Suppose $\lambda I_n=[X,Y]$ where $X,Y$ are ${\cal U}_k$-matrices. Then $X=[X,Y](YXY^{-1})=Y(\lambda X) Y^{-1}$. By Propositions \ref{prop:evalue}-\ref{prop:papinv}, $\lambda X$ is a ${\cal U}_k$-matrix and so $\lambda=1$ since $1$ is the only eigenvalue of $X.$
\end{proof}

    For any $P,X,Y\in GL_n(\field)$, note that $[X,Y]^{-1}=[Y,X]$ and $P[X,Y]P^{-1}=[PXP^{-1},PYP^{-1}]$. These imply the following.
    
    \begin{proposition} \label{prop:rmk2.1}
        Let $A\in SL_n(\mathbb F)$ and $P\in GL_n(\field)$. The following are equivalent:
        \begin{enumerate}[(i)]
        \item $A$ is a product of $r$ commutators of ${\cal U}_k$-matrices;
        \item $A^{-1}$ is a product of $r$ commutators of ${\cal U}_k$-matrices;
        \item $PAP^{-1}$ is a product of $r$ commutators of ${\cal U}_k$-matrices.
        \end{enumerate}
    \end{proposition}
\begin{proof} Suppose $A=\prod\limits_{i=1}^r[X_i,Y_i]$ where $X_i,Y_i$ are ${\cal U}_k$-matrices for all $i\in\{1,...,r\}$. Equivalently, $A^{-1}=\prod\limits_{i=1}^r[Y_{r+1-i},X_{r+1-i}]$, proving the equivalence of (i) and (ii). For the equivalence of (i) and (iii), observe that $PAP^{-1}=\prod\limits_{i=1}^r[PX_iP^{-1},PY_iP^{-1}].$ Since for each $i$, $PX_iP^{-1}$ and $PY_iP^{-1}$ are ${\cal U}_k$-matrices due to Proposition \ref{prop:papinv}, it follows that $PAP^{-1}$ is a product of $r$ commutators of ${\cal U}_k$-matrices. On the other hand, if $PAP^{-1}$ is a product of $r$ commutators of ${\cal U}_k$-matrices, then so is $A$ by an analogous argument applied to $P^{-1}(PAP^{-1})P=A$. 
\end{proof}
%We also provide some results on direct sums of matrices.
    
%\begin{proposition}\label{prop:rmk2.1a}
 %   If $A\in M_m(\mathbb F)$ and $B\in M_n(\mathbb F)$ are unipotent matrices of index $k_1$ and $k_2$, respectively, then $A\oplus B$ is a unipotent matrix of index $\max\{k_1,k_2\}$.
%\end{proposition}
%\begin{proof}
 %   Without loss of generality, suppose that $k_1\leq k_2$. It follows that $(A-I_m)^{k_2}=0_m$. Now observe that 
  %  \begin{center}
   %     $(A\oplus B-I_{m+n})^{k_2}=(A-I_m)^{k_2}\oplus(B-I_n)^{k_2}=0_{m+n}$.
   % \end{center}
    %Since $(B-I_n)^{k_2-1}\ne0_n$, we have that $(A\oplus B-I_{m+n})^{k_2-1}\ne0_{m+n}$. Therefore, $A\oplus B$ is a unipotent matrix of index $k_2=\max\{k_1,k_2\}$.
%\end{proof}

%\begin{proposition}\label{id}
%    Let $n,k\in\mathbb N$ such that $n,k>1$. The identity matrix $I_n$ is a commutator of unipotent matrices of index $k$.
%\end{proposition}

%If $m\in \mathbb N$ and $X,Y\in GL_n(\mathbb F)$, then $I_m\oplus [X,Y]=[I_m\oplus X, I_m\oplus Y].$ This observation is used to prove the following.

  Suppose $X,Y\in GL_m(\field)$ and $V,W\in GL_n(\field)$. Then $[X\oplus V,Y\oplus W]=[X,Y]\oplus [V,W]$. This observation is used to prove the following.
  
  %\begin{proposition}\label{prop:I_2+prod}
  %Let $A\in SL_m(\field)$ and $B\in SL_n(\field)$ be products of $r$ and $s$ commutators of ${\cal U}_k$-matrices, respectively. Then $A\oplus B$ is a product of at most $\max\{r,s\}$ commutators of ${\cal U}_k$-matrices. In particular, $A\oplus I_n$ and $I_m\oplus B$ are products of at most $r$ and $s$ commutators of ${\cal U}_k$-matrices, respectively.
  
%        For any $m\in \mathbb N$, $I_m\oplus A$ and $A\oplus I_m$ are products of $r$ commutators of ${\cal U}_k$-matrices provided that $A\in SL_n(\mathbb F)$ is a product of $r$ commutators of ${\cal U}_k$-matrices.
    
   % \end{proposition}
 %   \begin{proof}
 %      Since $A\oplus I_m$ is similar to $I_m\oplus A$, it suffices to prove the claim for $I_m\oplus A$ by Proposition \ref{prop:rmk2.1}. Suppose that $A=\prod\limits_{i=1}^r[X_i,Y_i]$ where $X_i,Y_i$ are ${\cal U}_k$-matrices for all $i\in\{1,...,r\}$. Observe that
  %          $I_m\oplus A=\prod\limits_{i=1}^r[I_m\oplus X_i,I_m\oplus Y_i]$
  %      and $I_m\oplus X_i$, $I_m\oplus Y_i$ are ${\cal U}_k$-matrices for all $i\in\{1,...,r\}$.
  %  \end{proof}

  %If $X,Y,W,Z$ are unipotent matrices of index $k$, then so are $X\oplus W$ and $Y\oplus Z$.
    
   % Then $[X,Y]\oplus[W,Z]=[X\oplus W,Y\oplus Z]$ is a commutator of unipotent matrices of index $\max\{k_1,k_2\}$. This follows from $[X,Y]\oplus [W,Z]=(X\oplus W)(Y\oplus Z)(X\oplus W)^{-1}(Y\oplus Z)^{-1}=[X\oplus W,Y\oplus Z]$ where $X\oplus W$, $Y\oplus Z$ are unipotents of index $\max\{k_1,k_2\}$ by Proposition \ref{prop:rmk2.1a}.
\begin{proposition} \label{prop:I_2+prod}
  Let $A\in SL_m(\field)$ and $B\in SL_n(\field)$ be products of $r$ and $s$ commutators of ${\cal U}_k$-matrices, respectively. Then $A\oplus B$ is a product of at most $\max\{r,s\}$ commutators of ${\cal U}_k$-matrices. In particular, $A\oplus I_n$ and $I_m\oplus B$ are products of at most $r$ and $s$ commutators of ${\cal U}_k$-matrices, respectively.
\end{proposition}
\begin{proof}
    Let $s,r\in\mathbb N$. Assume that $r\leq s$ (the argument for $r>s$ is analogous). Suppose that $A=\prod\limits_{i=1}^r[X_i,Y_i]$ and $B=\prod\limits_{j=1}^s[V_j,W_j]$ where $X_i,Y_i$ and $V_j,W_j$ are ${\cal U}_k$-matrices. Observe that
    \[\begin{array}{rcl}A\oplus B&=&\left(\prod\limits_{i=1}^r[X_i,Y_i]\oplus[V_i,W_i]\right)\left(\prod\limits_{i=r+1}^{s}I_m\oplus [V_i,W_i]\right)\\
    &=&\left(\prod\limits_{i=1}^r[X_i\oplus V_i,Y_i\oplus W_i]\right)\left(\prod\limits_{i=r+1}^{s}[I_m\oplus V_i, I_m\oplus ,W_i]\right)\end{array}\]
    where only the first factor above is present if $s=r$. For all $i$, $X_i\oplus V_i$, $Y_i\oplus W_i$, $I_m\oplus V_i$, and $I_m\oplus W_i$ are ${\cal U}_k$-matrices. Thus, $A\oplus B$ is a product of $r+(s-r)=s=\max\{r,s\}$ commutators of ${\cal U}_k$-matrices. The claim about $A\oplus I_n$ and $I_m\oplus B$ follows from the preceding and Proposition \ref{id}.
    \end{proof}

%We have the following results that give the orders of $GL_n(\field)$ and $SL_n(\field)$ for a finite field $\field$.
%\begin{proposition}\label{prop:|GL_n|}{\rm{\cite[Corollary 11.14]{dummit}}} 
 %   Let $|\field|=q$. Then, $$|GL_n(\field)|=(q^n-1)(q^n-q)(q^n-q^2)\cdots(q^n-q^{n-1}).$$
%\end{proposition}

%Notations from group theory are adapted from \cite{hungerford}. 

  % In particular, if $X$ is a unipotent of index $k$, then so is $X^{-1}$. This observation implies the following.]

%%%%%%%%%%%%%%%%%%%%%%%%%%%%%%%%%%%%%%%%%%%%%%%%%%%%%%%%%%%%%%%%%%%%%%
%%%%%%%%%%%%%%%%%%%%%%%%%%%%%%%%%%%%%%%%%%%%%%%%%%%%%%%%%%%%%%%%%%%%%%
%%%%%%%%%%%%%%%%%%%%%%%%%%%SL_2%%%%%%%%%%%%%%%%%%%%%%%%%%%%%%%%%%%%%%%
%%%%%%%%%%%%%%%%%%%%%%%%%%%%%%%%%%%%%%%%%%%%%%%%%%%%%%%%%%%%%%%%%%%%%%
%%%%%%%%%%%%%%%%%%%%%%%%%%%%%%%%%%%%%%%%%%%%%%%%%%%%%%%%%%%%%%%%%%%%%%
\section{Decomposition of matrices in \texorpdfstring{$SL_2(\field)$ into commutators}{}} \label{sec:SL2}
%This section discusses properties and commutators of $2\times 2$ unipotent matrices of index 2.

Suppose $N\in M_2(\mathbb F)$ is nonzero and has zero trace. Write $N=\begin{bmatrix}a& b\\ c& -a\end{bmatrix}$ for some $a,b,c\in\mathbb F$ not all zero. Since $N^2=(a^2+bc)I_2$, it follows that such $N$ is a nilpotent matrix of index $2$ if and only if $a^2+bc=0$. This characterization implies the following.

    \begin{proposition} \label{rmk:uniform}
       Let $A\in SL_2(\field)$. Then $A$ is a ${\cal U}_2$-matrix if and only if
       $ A=\begin{bmatrix}
            1+a & b \\
            c & 1-a
        \end{bmatrix}
       $
       where $a,b,c\in\mathbb F$ are not all zero and $a^2+bc=0$. %In particular, either $abc\neq 0$ or $a=0$ and $bc\neq 0$.
    \end{proposition}
    
Proposition \ref{rmk:uniform} suggests two types of ${\cal U}_2$-matrices based on the condition $a^2+bc=0$. Corresponding to $a=0$, a \textit{type \textup(i\textup)} is of the form $\begin{bmatrix}
            1 & b \\ 0 & 1
        \end{bmatrix}$ or $\begin{bmatrix}
            1 & 0 \\ c & 1
        \end{bmatrix}$ where $b,c\neq 0$. Corresponding to $a\neq 0$, a \textit{type \textup(ii\textup)} is of the form $\begin{bmatrix}
            1+a & b \\ c & 1-a
        \end{bmatrix}$ where $a,b,c\neq 0$ and $a^2+bc=0$.

 %   Let $A\in M_2(\field)$ such that $(A-I_2)^2=0$ and $A-I_2\ne0$. It follows from Proposition \ref{prop:evalue} that the Jordan canonical form of $A$ is $J_2(1)$.

%    \begin{proposition} \label{prop:unisim}
 %       The Jordan canonical form of a $2\times2$ unipotent matrix of index 2 is $J_2(1)$.
  %  \end{proposition}
    
    \begin{theorem}\label{prop:trace}
       Let $A\in SL_2(\mathbb F)$ be a nonscalar matrix. Then $A$ is a commutator of ${\cal U}_2$-matrices if and only if $\textup{tr}(A)=2+\alpha^2$ for some nonzero $\alpha\in\mathbb F$. %If $A$ is scalar, then $\tr(A)=2$. If $A$ is nonscalar, then $\tr(A)=2+a^2$ for some $a\in\field\setminus\{0\}$.
    \end{theorem}
    \begin{proof}
        Suppose $A=[X,Y]$ where $X,Y$ are ${\cal U}_2$-matrices. By Proposition \ref{prop:evalue}, $PXP^{-1}=J_2(1)$ for some $P\in GL_2(\mathbb F)$, and so \[PAP^{-1}=P[X,Y]P^{-1}=[PXP^{-1},PYP^{-1}]=[J_2(1),PYP^{-1}].\] By Propositions \ref{prop:papinv} and \ref{rmk:uniform}, $PYP^{-1}$ is either a type (i) or type (ii) ${\cal U}_2$-matrix. If $PYP^{-1}$ is type (ii), i.e., $PYP^{-1}=\begin{bmatrix}
                1+a & b \\ c & 1-a
            \end{bmatrix}$ for some $a,b,c\ne0$ such that $a^2+bc=0$, then  \[PAP^{-1}=[J_2(1),PYP^{-1}]=\begin{bmatrix}
                    c^2+c+ac+1 & -(c+ac+a^2+2a)\\
                    c^2 & -ac-c+1
                \end{bmatrix}.\] Hence, $\textup{tr}(A)=2+c^2$ where $c\neq0$. Suppose $PYP^{-1} $ is type (i), i.e., $PYP^{-1}$ is either $\begin{bmatrix}
            1 & b \\ 0 & 1
        \end{bmatrix}$ or $\begin{bmatrix}
            1 & 0 \\ c & 1
        \end{bmatrix}$ where $b,c\neq 0$. In the first case, $PAP^{-1}=[J_2(1),PYP^{-1}]=I_2,$ and so $A=I_2$, contradicting the assumption. In the second case, \[PAP^{-1}=\begin{bmatrix}
                    c^2+c+1 & -c\\
                    c^2 & -c+1
                \end{bmatrix}\] which implies  $\textup{tr}(A)=2+c^2$ where $c\neq0$.

                 Conversely, suppose $\tr(A)=2+\alpha^2$ for some nonzero $\alpha\in\field$. Since $A\in SL_2(\mathbb F)$ is a nonscalar matrix, the minimal polynomial $m_A(x)=p_A(x)=x^2-\mbox{tr}(A)x+1.$ If $m_A(x)$ is irreducible, then $A$ is similar to the companion matrix $\begin{bmatrix}0& -1\\ 1& \mbox{tr}(A)\end{bmatrix}$. Otherwise, $m_A(x)=(x-\lambda)(x-\lambda^{-1})$ for some $\lambda\in \mathbb F\setminus\{0\}$. By assumption $\alpha\neq 0$, and so $\lambda\neq 1$. If $\lambda=-1,$ then $A$ is similar to the companion matrix $\begin{bmatrix}0& -1\\ 1& - 2\end{bmatrix}$ since $A$ is a nonscalar matrix. If $\lambda\neq \pm 1$, then $A$ is similar to $\textup{diag}(\lambda,\lambda^{-1})$; note that $A$ is similar to $P\textup{diag}(\lambda,\lambda^{-1})P^{-1}=\begin{bmatrix}0& -1\\ 1&\lambda+\lambda^{-1}\end{bmatrix}$ where $P=\begin{bmatrix} 1& -1\\ -\lambda& \lambda^{-1}\end{bmatrix}$. In any case, $A$ is similar to $B:=\begin{bmatrix}0& -1\\ 1& \mbox{tr}(A)\end{bmatrix}=\begin{bmatrix}0& -1\\ 1& 2+\alpha^2\end{bmatrix}$. By Proposition \ref{prop:rmk2.1}, it suffices to show that $B$ is a commutator of ${\cal U}_2$-matrices. Direct computations reveal that $B=[X,Y]$ where \[X=\begin{bmatrix}1& \alpha^2\\ 0 & 1\end{bmatrix}\ \textup{and}\ Y=\begin{bmatrix}1-\alpha^{-1}(\alpha^2+\alpha+1)& -\alpha^{-1}(\alpha^2+\alpha+1)^2\\ \alpha^{-1}& 1+\alpha^{-1}(\alpha^2+\alpha +1)\end{bmatrix}\]
        Note that $X$ is a type (i) ${\cal U}_2$-matrix while $Y$ is either a type (i) or (ii) ${\cal U}_2$-matrix ($Y$ is a type (i) ${\cal U}_2$-matrix if and only if $\alpha^2+\alpha+1=0$).

   \end{proof}

\begin{corollary}\label{thm:J2(1)}
Any nonscalar matrix $A\in SL_2(\mathbb F)$ that has $\tr(A)=2$ is not a commutator of ${\cal U}_2$-matrices. In particular, $J_2(1)$ is not a commutator of $\mathcal U_2$-matrices.
\end{corollary}

\begin{corollary} \label{thm:J2(-1)}
 Let $\textup{char}(\mathbb F)\ne2$. Then $J_2(-1)\in SL_2(\field)$ is a commutator of $\mathcal U_2$-matrices if and only if $-1=a^2$ for some $a\in\mathbb F$.
\end{corollary}
    \begin{proof}
         By Theorem \ref{prop:trace}, $J_2(-1)$ is a commutator of $\mathcal U_2$-matrices if and only if  $-2=\textup{tr}(J_2(-1))=2+\alpha^2$ for some nonzero $\alpha\in\mathbb F$. Equivalently, $-1=(2^{-1}\alpha)^2$ since $\chara(\mathbb F)\neq 2$.
    \end{proof}

   % \begin{proposition} \label{thm:J2(1)}
    %     The matrix $J_2(1)\in SL_2(\field)$ is not a commutator of $\mathcal U_2$-matrices.
     %\end{proposition}
    % \begin{proof}
     %    Suppose $J_2(1)=[X,Y]$
     %    where $X=\begin{bmatrix}
      %       1+a & b \\ c & 1-a
      %   \end{bmatrix}$, $Y=\begin{bmatrix}
      %      1+x & y \\ z & 1-x
      %   \end{bmatrix}$ for some $a,b,c$ not all zero, $x,y,z$ not all zero, and $a^2+bc=x^2+yz=0$. Observe that
      %   \begin{center}
      %       $\begin{bmatrix}
      %           1+x+z & y+1-x \\ z & 1-x
      %       \end{bmatrix}=J_2(1)Y=XYX^{-1}$.
      %   \end{center}
      %  Taking the trace yields $2+z=\tr(J_2(1)Y)=\tr(Y)=2$, and so $z=0$. Since $x^2+yz=0$, $x=0$. Similarly, computing the trace of $X^{-1}J_2(1)=YX^{-1}Y^{-1}$ yields $c=0$, and so $a=0$ due to $a^2+bc=0$. Hence, $X=\begin{bmatrix}
      %       1 & b \\ 0 & 1
      %   \end{bmatrix}$ and $Y=\begin{bmatrix}
       %      1 & y \\ 0 & 1
       %  \end{bmatrix}$ which implies that $J_2(1)=[X,Y]=I_2$, a contradiction.
     %\end{proof}

From Proposition \ref{id}, $I_2$ is the \textit{only} scalar matrix in $SL_2(\mathbb F)$ which is a commutator of ${\cal U}_2$-matrices. In particular, when $\chara(\mathbb F)\neq 2$, $-I_2$ is \textit{not} a commutator of ${\cal U}_2$-matrices. It was observed recently 
\cite[Theorem 8.2]{bien} that writing $-I_2$ as a product of two commutators of ${\cal U}_2$-matrices is equivalent to a sum-of-squares problem in $\mathbb F$. A reformulation and an alternative proof are presented here. %The version below does not assume that $|\mathbb F|\geq 4 $ as opposed to the one in\cite[Theorem 8.2]{bien}. 

%A reformulation and an alternative proof are presented here where .

    \begin{corollary} \label{cor:-I_2}
        Let $\textup{char}(\mathbb F)\ne2$. Then $-I_2$ is a product of two commutators of $\mathcal U_2$-matrices if and only if $-1=a^2+b^2$ for some nonzero $a,b\in\field$.
    \end{corollary}
    \begin{proof} 
        
    %    Since $b^2\ne b^{-2}$, Corollary \ref{lem:hou} and Proposition \ref{prop:rmk2.1} imply that first factor is a commutator of $\mathcal U_2$-matrices. On the other hand, the second factor is a product of two commutators of $\mathcal U_2$-matrices by Theorem \ref{thm:SL2}. Hence, $-I_2$ is a product of three commutators of $\mathcal U_2$-matrices.

 Suppose that $-I_2=AB$ where $A,B$ are commutators of $\mathcal U_2$-matrices. Since $\chara(\mathbb F)\neq 2$, $A,B$ are nonscalar matrices. By Proposition \ref{prop:rmk2.1}, Theorem \ref{prop:trace}, and the assumption, there exist $\alpha,\beta\in\field\setminus\{0\}$ such that \[-2-\alpha^2=\tr(-A^{-1})=\tr(B)=2+\beta^2.\] If $a:=2^{-1}\alpha$ and $ b:=2^{-1}\beta$, then $a,b\in\mathbb{F}\setminus\{0\}$ and \[a^2+b^2=2^{-2}\alpha^2+2^{-2}\beta^2=2^{-2}(\alpha^2+\beta^2)=-1.\]

            Conversely, suppose there exist $a,b\in\field\setminus\{0\}$ such that $a^2+b^2=-1$. If $\alpha:=2a$ and $\beta:=2b$, then $\alpha,\beta\in\mathbb F\setminus\{0\}$ and $-I_2=AB$ where $A=\begin{bmatrix}
                2 & 1\\
                2\alpha^2-1 & \alpha^2
            \end{bmatrix}$ and $B=\begin{bmatrix}
                -\alpha^2 & 1\\
                2\alpha^2-1 & -2
            \end{bmatrix}$. Since $\tr(A)=2+\alpha^2$ and $\tr(B)=-2-\alpha^2=2+\beta^2$ where $\alpha,\beta\in\field\setminus\{0\}$, $A$ and $B$ are commutators of $\mathcal U_2$-matrices by Theorem \ref{prop:trace}.
    \end{proof}

   Consider $-I_2\in SL_2(\mathbb F)$ where $\field=\reals$ or $|\field|=5$. Proposition \ref{id} and Corollary \ref{cor:-I_2} imply that $-I_2$ is not a commutator nor a product of two commutators of $\mathcal U_2$-matrices. This does not automatically mean that $-I_2$ is \textit{not} a product of commutators of ${\cal U}_2$-matrices. This issue will be addressed in Proposition \ref{thm:-I_2}. 

   For the meantime, we end this part with a characterization of nonscalar diagonal matrices in $SL_2(\mathbb F)$ which are commutators of ${\cal U}_2$-matrices. 
   
 %  Some techniques of the proof are borrowed from \cite{hou}.
  \begin{corollary} \label{lem:hou}
         Let $a\in\field\setminus\{-1,0,1\}$. Then $\diag(a,a^{-1})$ is a commutator of $\mathcal U_2$-matrices if and only if $a=b^2$ for some $b\in\mathbb F$.
    \end{corollary}
   
   %  \begin{corollary} \label{lem:hou}
   %      Let $a\in\field\setminus\{-1,0,1\}$.
   %      \begin{enumerate}
   %          \item[\textup(i\textup)] If $\textup{char}(\mathbb F)\ne2$, then $\diag(a,a^{-1})$ is a commutator of $\mathcal U_2$-matrices if and only if $a\ne-1$ and $a=b^2$ for some $b\in\mathbb F$;
   %          \item[\textup(ii\textup)] If $\textup{char}(\mathbb F)=2$, then $\diag(a,a^{-1})$ is a commutator of $\mathcal U_2$-matrices. 
   %      \end{enumerate}
   % \end{corollary}
    \begin{proof} 
    %Observe that $A$ is a scalar matrix if and only if $a=\pm 1$. By Proposition \ref{id} and remarks before it, $A$ is a scalar matrix and a commutator of ${\cal U}_2$-matrices if and only if $a=1$. Assume $a\neq 1$. 

    %Assume $\chara(\field)\ne2$. 
        Observe that since $a\in \mathbb F\setminus\{-1,0,1\}$, $\diag(a,a^{-1})$ is a nonscalar matrix in $SL_2(\mathbb F)$. 

If $\diag(a,a^{-1})$ is a commutator of $\mathcal U_2$-matrices, then Theorem \ref{prop:trace} ensures that $a+a^{-1}=\tr(\diag(a,a^{-1}))=2+\alpha^2$ for some nonzero $\alpha\in\mathbb F$. Since $\alpha\neq 0$ and $a\neq a^{-1}$, $b:=(a-1)\alpha^{-1}\in \mathbb F\setminus\{0\}$ and \[b^2=(a-1)^2\alpha^{-2}=(a^2-2a+1)\alpha^{-2}=a(a+a^{-1}-2)\alpha^{-2}=a.\]
        
        Conversely, suppose $a=b^2$ for some $b\in\field$. By assumption on $a$, it is necessary that $b\in\mathbb F\setminus\{-1,0,1\}$. Hence, $\alpha:=b-b^{-1}\in\mathbb F\setminus\{ 0\}$ and \[\mbox{tr}(\diag(a,a^{-1}))=a+a^{-1}=2+(b-b^{-1})^2=2+\alpha^2.\]
       By Theorem \ref{prop:trace}, $\diag(a,a^{-1})$ is a commutator of ${\cal U}_2$-matrices.

     %   \item[(i)] Assume $\chara(\field)\ne2$. 
     %   Suppose $a\ne-1$ and $a=b^2$ for some $b\in\field$. The case $a=1$ follows from Proposition \ref{id}. Assume $a\ne1$. Observe that $\alpha:=b-b^{-1}\neq 0$ and \begin{equation}\label{ktrace}\mbox{tr}(A)=a+a^{-1}=2+(b-b^{-1})^2=2+\alpha^2.\end{equation}
     %  Since $A$ is nonscalar and \eqref{ktrace}, $A$ is a commutator of ${\cal U}_2$-matrices due to Theorem \ref{prop:trace}.

%     Theorem \ref{prop:trace} ensures that $a+a^{-1}=\tr(A)=2+\alpha^2$ for some nonzero $\alpha\in\mathbb F$. Note that $b:=(a-1)\alpha^{-1}\neq 0$ and \[b^2=(a-1)^2\alpha^{-2}=(a^2-2a+1)\alpha^{-2}=a(a+a^{-1}-2)\alpha^{-2}=a.\]

  \end{proof}

    As we will see, there is stark contrast in the results for fields with $|\mathbb F|\leq 3$ and $|\mathbb F|\geq 4$.
  
    \subsection{Fields with at most three elements}
    %Notations in group theory are adapted from \cite{hungerford}. Let $G$ be a group and $\varnothing\ne H,N\subseteq G$. We denote by $H\leq G$ ($H<G$) if $H$ is a subgroup (proper subgroup) of $G$. We also denote by $G'$ the commutator subgroup of $G$, that is, the subgroup generated by all commutators of $G$. If $N$ is a normal subgroup of $G$, then we denote it by $N\trianglelefteq G$ with the factor group given by $G/N$. Moreover, the index of $H$ in $G$ is given by $[G:H]$ and the normalizer of $H$ in $G$ is given by $N_G(H)=\{g\in G\ |\ gHg^{-1}=H\}$.

   This subsection involves group theoretic considerations, and so recalling notations and facts is necessary. Let $G$ be a group. A subgroup (proper subgroup) $H$ of $G$ is denoted by $H\leq G$ ($H<G$). If $S$ is a subset of $G$, then it generates a subgroup which is denoted by $\langle S\rangle.$ A normal subgroup $H$ of $G$ is denoted by $H\trianglelefteq G$ while the factor group it defines is denoted by $G/H$. The index of a subgroup $H$ in $G$ is denoted by $[G:H]$ and the normalizer of $H$ in $G$ is denoted by $N_G(H)$. Analogous to the matrix case, a \textit{commutator} in $G$ is an element of the form $[x,y]:=xyx^{-1}y^{-1}$ for some $x,y\in G$. The subgroup generated by all commutators in $G$ is called the \textit{derived or commutator subgroup of} $G$ and is denoted by $G'$.
   
   Observe that \begin{equation}\label{prop:|SL_n|}
|SL_n(\mathbb F)|=\frac{1}{q-1}\prod\limits_{i=0}^{n-1}(q^n-q^i)\ \textup{whenever}\ |\mathbb F|=q.
    %$|SL_n(\mathbb F)|=\dfrac{\prod_{i=0}^{n-1}(q^n-q^i)}{q-1}.$
    %\[|SL_n(\field)|=\frac{(q^n-1)(q^n-q)(q^n-q^2)\cdots(q^n-q^{n-1})}{q-1}.\]
\end{equation}
Indeed, if  $\phi:GL_n(\field)\to\field\setminus\{0\}$ is defined by $\phi(A)=\det(A)$, then the First Isomorphism Theorem for groups implies \[GL_n(\field)/SL_n(\field)=GL_n(\field)/\ker(\phi)\cong\img(\phi)=\field\setminus\{0\}.\] Since $|GL_n(\field)|=\prod\limits_{i=0}^{n-1}(q^n-q^i)$
by \cite[Corollary 11.2.14]{dummit}, \eqref{prop:|SL_n|} holds. 

    Dickson showed in \cite{dickson} that $SL_n(\field)'=SL_n(\field)$ except for $SL_2(\field)$ with $|\field|\leq3$. In this subsection, we confirm Dickson's result for $SL_2(\mathbb F)$ with $|\field|\leq3$ by providing an alternative proof using Sylow theorems and results on commutator subgroups.
    
   % of $SL_2(\field)'<SL_2(\field)$
  %  . To prove the claim, Sylow theorems and results on commutator subgroups are used. We believe that this approach is more direct. %Finally, we show that every matrix in $SL_2(\field)'$ is a product of at most two commutators of unipotent matrices of index 2 and two is the smallest such number. 
    %We start by finding normal subgroups of $SL_2(\field)$ where $|\field|\leq3$.

    \begin{lemma}\label{lem:normalF2}
        Let $|\field|=2$. If $A:=\begin{bmatrix}
        0 & 1\\
        1 & 1
    \end{bmatrix}\in SL_2(\mathbb F), $ then $SL_2(\mathbb F)'\subseteq \langle A\rangle.$
    %$\langle A\rangle \trianglelefteq SL_2(\field)$. 
    \end{lemma}
\begin{proof}
 Since $\textup{char}\mathbb (\mathbb F)=2$, $A^2\neq I_2$ but $A^3=I_2$. Hence, $[SL_2(\field):\langle A\rangle ]=\frac{|SL_2(\field)|}{|\langle A\rangle|}=2$ due to \eqref{prop:|SL_n|}. By \cite[Exercise I.5.1]{hungerford}, $\langle A\rangle \trianglelefteq SL_2(\field)$. In particular, $SL_2(\field)/\langle A\rangle$ is abelian, and so $SL_2(\field)'\subseteq\langle A\rangle$ by \cite[Theorem II.7.8]{hungerford}. \end{proof}
% The claim follows from \cite[Exercise I.5.1]{hungerford}.

  %  We also find a normal subgroup of $SL_2(\field)$ where $|\field|=3$.
Let $|\mathbb F|=3$ and consider the following elements of $SL_2(\mathbb F)$:
\begin{equation}\label{quatmat}
    \mathbf i:=\begin{bmatrix}
        1 & 1\\
        1 & 2
    \end{bmatrix}, \ \mathbf j:=\begin{bmatrix}
        0 & 1\\
        2 & 0
    \end{bmatrix},\ \textup{and}\ \mathbf k:=\begin{bmatrix}
        2 & 1\\
        1 & 1
    \end{bmatrix}.
\end{equation}
The matrices in \eqref{quatmat} satisfy
\begin{center}$\mathbf{i}^2=\mathbf j^2=\mathbf k^2=-I_2$, $\mathbf{ij}=-\mathbf{ji}=\mathbf k$, $\mathbf{jk}=-\mathbf{kj}=\mathbf i$, and $\mathbf{ki}=-\mathbf{ik}=\mathbf j$.\end{center} 
Thus, $\{I_2,\mathbf i,\mathbf j,\mathbf k,-I_2,-\mathbf i,-\mathbf j,-\mathbf k\}=\langle \{\mathbf{i}, \mathbf{j},\mathbf{k}\}\rangle=:{\cal Q}$ is the quaternion group.

%  Let ${\cal Q}=\{I_2,\mathbf i,\mathbf j,\mathbf k,-I_2,-\mathbf i,-\mathbf j,-\mathbf k\}$ where $\mathbf i=\begin{bmatrix}
 %       1 & 1\\
  %      1 & 2
  %  \end{bmatrix}$, $\mathbf j=\begin{bmatrix}
  %      0 & 1\\
  %      2 & 0
  %  \end{bmatrix}$, and $\mathbf k=\begin{bmatrix}
  %      2 & 1\\
  %      1 & 1
  %  \end{bmatrix}$
%The elements of ${\cal Q}$ in the next result satisfy
%\begin{center}$\mathbf{i}^2=\mathbf j^2=\mathbf k^2=-I_2$, $\mathbf{ij}=-\mathbf{ji}=\mathbf k$, $\mathbf{jk}=-\mathbf{kj}=\mathbf i$, and $\mathbf{ki}=-\mathbf{ik}=\mathbf j$.\end{center} 
\begin{lemma}\label{lem:normalF3}
    Let $|\field|=3$. If ${\cal Q}:=\langle \{\mathbf{i}, \mathbf{j},\mathbf{k}\}\rangle$ where $\mathbf{i},\mathbf{j},\mathbf{k}\in SL_2(\mathbb F)$ is as defined in \eqref{quatmat}, then $SL_2(\field)'\subseteq \mathcal Q$.
    %Then ${\cal Q}\trianglelefteq SL_2(\field)$.
\end{lemma}
\begin{proof} By \eqref{prop:|SL_n|}, $|SL_2(\field)|=24=2^3\cdot3$.
Since $|{\cal Q}|=2^3$, \cite[Corollary II.5.8]{hungerford}(i) guarantees that ${\cal Q}$ is a Sylow 2-subgroup of $SL_2(\field)$. If it can be established that ${\cal Q}\trianglelefteq SL_2(\field)$, then the claim follows since $SL_2(\field)/\mathcal Q$ is abelian, and thus $SL_2(\field)'\subseteq \mathcal Q$ by \cite[Theorem II.7.8]{hungerford}.

To show that ${\cal Q}\trianglelefteq SL_2(\field)$, it suffices to prove that $SL_2(\mathbb F)$ has a unique Sylow 2-subgroup due to \cite[Corollary II.5.8]{hungerford}(iii). Suppose not. We arrive at a contradiction by showing that all non-identity elements of $SL_2(\mathbb F)$ can only have orders which are powers of $2$ or $3$; we then exhibit an element whose order is $6.$

By the Third Sylow Theorem, $SL_2(\mathbb F)$ has three Sylow 2-subgroups, say $P_1,P_2,P_3$. We claim that $|P_1\cap P_2\cap P_3|=4$. Let $\mathcal P=\{P_1,P_2,P_3\}$. For $P\in\mathcal P$, define $\mathcal S_P$ be the set of all subgroups of $SL_2(\field)$ conjugate to $P$. Since each element of ${\cal S}_P$ is a Sylow 2-subgroup due to \cite[Corollary II.5.8]{hungerford}(ii), $\mathcal S_P\subseteq\mathcal P$. Equality holds since the Second Sylow Theorem ensures that $\mathcal P\subseteq\mathcal S_P$. Hence, \[3=|{\cal P}|=|{\cal S}_P|=[SL_2(\field):N_{SL_2(\field)}(P)]=\frac{24}{|N_{SL_2(\field)}(P)|}\] due to \cite[Corollary II.4.4]{hungerford}, and so $|N_{SL_2(\field)}(P)|=8$. Since $P\trianglelefteq N_{SL_2(\field)}(P)$ and $|P|=8$ by \cite[Corollary II.5.8]{hungerford}(i), it follows that $P=N_{SL_2(\field)}(P)$. Consequently, $P_1\cap P_2\cap P_3=\bigcap_{P\in {\cal P}}N_{SL_2(\field)}(P).$ Now,
let $\Sigma(\mathcal P)=\{\sigma\,|\, \sigma$ is a bijection on $\mathcal P\}$. For a given $A\in SL_2(\mathbb F)$, define $\phi(A)(P)=APA^{-1}$ for all $P\in {\cal P}$. The mapping $\phi: SL_2(\field)\to\Sigma(\mathcal P)$ is a group homomorphism (see \cite[Theorem II.4.5]{hungerford}). Note that $A\in\ker(\phi)$ if and only if $A\in SL_2(\field)$ such that $APA^{-1}=\phi(A)(P)=P$ for all $P\in\mathcal P$. Equivalently, $A\in N_{SL_2(\field)}(P)$ for all $P\in\mathcal P$. It follows that $\ker(\phi)=\bigcap_{P\in {\cal P}}N_{SL_2(\field)}(P)=P_1\cap P_2\cap P_3$. By the First Isomorphism Theorem for groups,
\[|P_1\cap P_2\cap P_3|=|\ker(\phi)|=\frac{|SL_2(\field)|}{|\textup{Im}(\phi)|}\geq \frac{|SL_2(\field)|}{|\Sigma({\cal P})|}=\frac{24}{6}=4.\]
    
   % Let $P\in\mathcal P$ and let $\mathcal S$ be the set of all subgroups of $SL_2(\field)$ conjugate to $P$. Since each element of ${\cal S}$ is a Sylow 2-subgroup due to \cite[Corollary II.5.8]{hungerford}(ii), $\mathcal S\subseteq\mathcal P$. Equality holds since the Second Sylow Theorem ensures that $\mathcal P\subseteq\mathcal S$. Hence, $|{\cal P}|=|{\cal S}|=[SL_2(\field):N_{SL_2(\field)}(P)]$ due to \cite[Corollary II.4.4]{hungerford}. Consequently, $|N_{SL_2(\field)}(P)|=8$. By \cite[Corollary II.5.8]{hungerford}(i), we have that $|P|=2^3=8$. Since $P\trianglelefteq N_{SL_2(\field)}(P)$, it further implies that $P=N_{SL_2(\field)}(P)$ for all $P\in\mathcal P$.
     %   Let $\mathcal S$ be the set of all subgroups of $SL_2(\field)$ conjugate to $P$. Then $|\mathcal S|=[SL_2(\field):N_{SL_2(\field)}(P)]$. The Second Sylow Theorem guarantees that $\mathcal P\subseteq\mathcal S$ while \cite[Corollary II.5.8]{hungerford}(ii) implies $\mathcal S\subseteq\mathcal P$. Therefore, $[SL_2(\field):N_{SL_2(\field)}(P)]=|\mathcal P|$ which further implies that $|N_{SL_2(\field)}(P)|=8$. By \cite[Corollary II.5.8]{hungerford}(i), we have that $|P|=2^3=8$. Since $P\trianglelefteq N_{SL_2(\field)}(P)$, it further implies that $P=N_{SL_2(\field)}(P)$ for all $P\in\mathcal P$.
  %  We now have that
   % \[
   % |P_1\cap P_2\cap P_3|=\left|\displaystyle\bigcap_{P\in {\cal P}}N_{SL_2(\field)}(P)\right|=|\ker(\phi)|\geq4.
    %\]
   \noindent To show equality, observe that $4\leq|P_1\cap P_2\cap P_3|\leq|P_i\cap P_j|\leq |P_i|= 8$. For distinct $P_i, P_j$, $|P_i\cap P_j|=4$ by Lagrange's Theorem. Therefore, $|P_1\cap P_2\cap P_3|=4$ which proves the claim. Thus, there are $4|{\cal P}|+|P_1\cap P_2\cap P_3|-1=15$ non-identity elements in $SL_2(\field)$ that have order a power of 2. 

 Since $|\langle J_2(1)\rangle|=|\langle J_2(1)^\top\rangle|=3$, both $\langle J_2(1)\rangle$ and $\langle J_2(1)^\top\rangle$ are Sylow $3$-subgroups of $SL_2(\field)$ due to \cite[Corollary II.5.8]{hungerford}(i). By the Third Sylow Theorem, $SL_2(\mathbb F)$ has exactly four Sylow 3-subgroups, say $R_1,R_2,R_3,R_4$. Since $|R_i|=3$ for all $i$, $R_i\cap R_j=\{I_2\}$ for all $i\ne j$. As a consequence, there are exactly 8 non-identity elements in $SL_2(\field)$ that have order a power of 3.

Since $|SL_2(\field)|=24=15+8+1$, all non-identity elements of $SL_2(\field)$ can only have orders which are powers of 2 or 3. However, observe that $J_2(-1)\in SL_2(\field)$ has order 6. Therefore, there must be a unique Sylow-2 subgroup of $SL_2(\field)$ as desired.\end{proof}

%These results lead us to the main theorem of this subsection. 
Lemmas \ref{lem:normalF2}-\ref{lem:normalF3} are used to prove the exceptional cases in Dickson's result.

\begin{theorem} \label{thm:Z2Z3}
Let $|\mathbb F|\leq 3$. Every matrix in $SL_2(\field)'$ is a product of at most $|\mathbb F|-1$ commutator of ${\cal U}_2$-matrices. The upper bound $|\mathbb F|-1$ cannot be further reduced. Moreover, $SL_2(\field)'<SL_2(\field)$. 
\end{theorem}
%Two is the least upper bound to the number of commutator of ${\cal U}_2$-matrix factors over all $\mathbb F$ with $|\mathbb F|\leq 3$.
\begin{proof}
    Suppose $|\field|=2$. If $A$ is as in Lemma \ref{lem:normalF2}, then $SL_2(\field)'\subseteq\langle A\rangle$. Since $\chara(\mathbb F)=2$ and $\tr(A)=\tr(A^2)=1=2+1^2$, it follows that $A$ and $A^2$ are commutators of ${\cal U}_2$-matrices by Theorem \ref{prop:trace}. %to direct calculations revealing that  $A=\left[J_2(1)^\top,J_2(1)\right]$ and $A^2=\left[J_2(1),J_2(1)^\top\right]$.
    Hence, $SL_2(\field)'=\langle A\rangle<SL_2(\field)$. %Since $J_2(1)$ and $J_2(1)^\top$ are ${\cal U}_2$-matrices, the factorizations also 
    In particular, every matrix in $SL_2(\field)'$ is a commutator of ${\cal U}_2$-matrices.

    Suppose $|\field|=3$. If ${\cal Q}$ is as in Lemma \ref{lem:normalF3}, then  $SL_2(\field)'\subseteq \mathcal Q$.
Equality holds by showing that every element of $\mathcal Q$ can be written as products of commutators of ${\cal U}_2$-matrices (necessarily, elements of $SL_2(\field)$). Note that all elements of ${\cal Q}\setminus\{-I_2\}$ are commutators of ${\cal U}_2$-matrices due to Proposition \ref{id} and Theorem \ref{prop:trace}. On the other hand, Corollary \ref{cor:-I_2} ensures that $-I_2$ is a product of two commutators of ${\cal U}_2$-matrices since $\textup{char}(\mathbb F)=3$ and $-1=2=a^2+b^2$ where $a=b=1$.
%For $-I_2$, observe that $-I_2=[X,Y][Y^{-1},X^{-1}]$ where $X=[J_2(1)^\top]^{-1}$ and %$Y=\begin{bmatrix}
%        0 & 2\\
%        1 & 2
%    \end{bmatrix}$ (both are ${\cal U}_2$-matrices by Proposition \ref{rmk:uniform}). 
   %If $-I_2$ is a commutator of unipotent matrices of index 2, then by Proposition \ref{prop:trace}, $-2=\textup{tr}(-I_2)=2+a^2$ for some $a\in\field$. Hence, $2=a^2$, a contradiction since no such $a$ exists in $\field$.
    Thus, every matrix in $\mathcal Q$ is a product of at most two commutators of ${\cal U}_2$-matrices and two is the smallest such number since $-I_2$ is not a commutator of ${\cal U}_2$-matrices due to Proposition \ref{id}. Consequently, $SL_2(\field)'=\mathcal Q<SL_2(\field)$.\end{proof}

%Now, the minimal polynomials of $\mathbf{i},\mathbf{j},\mathbf{k}$ are all equal to $x^2+1$ which is irreducible over $\field$, and hence the rational canonical forms of $\mathbf{i},\mathbf{j},\mathbf{k}$ are all equal to %$\begin{bmatrix}
%        0 & -1\\
%        1 & 0
 %   \end{bmatrix}$. Note that $\begin{bmatrix}
  %      0 & -1\\
  %      1 & 0
  %  \end{bmatrix}=\left[J_2(1),J_2(1)^\top\right]$, and so every nonscalar matrix in $\mathcal Q$ is a commutator of ${\cal U}_2$-matrices by Proposition \ref{prop:rmk2.1}. 

When $|\mathbb F|\leq 3$, Theorem \ref{thm:Z2Z3} implies that there are elements of $SL_2(\mathbb F)$, e.g., $J_2(1)$ and $ J_2(1)^\top$, which are not products of commutators of ${\cal U}_2$-matrices or matrices in $SL_2(\field)$. %Since ${\cal U}_2$-matrices are in $SL_2(\field)$, the theorem further implies that there are matrices in $SL_2(\field)$ that cannot be written as a product of commutators of ${\cal U}_2$-matrices. %In fact, more is true: if a matrix in $SL_2(\field)$ can be written written as a finite product of commutators of matrices in $SL_2(\field)$ (or an element of $SL_2(\field)'$), then it can also be written as a finite product of commutators of unipotent matrices of index 2.

\subsection{Fields with at least four elements}
The situation when $|\field|\geq4$ is different. We show that every matrix in $SL_2(\field)$ is a product of at most three commutators of $\mathcal U_2$-matrices. 
 
 %We first discuss the factorization of some special matrices in $SL_2(\field)$ such as $J_2(1)$, $J_2(-1)$, and $\diag(a,a^{-1})$ where $a\in\field\setminus\{0\}$.

   % We say that $a\in\field$ is a \textit{\textup(perfect square\textup) square in $\field$} if there exists $b\in\field$ such that $a=b^2$. 
   %A \textit{(perfect) square in} $\mathbb F$ is an element $a$ of the form $a=b^2$ for some $b\in \mathbb F$. If $\field$ is algebraically closed, then all elements of $\field$ are perfect squares in $\field$. 

%We consider the following lemma which is a result from \cite{bien}.
%    \begin{lemma}\cite[Lemma 8.1]{bien} \label{lem:bien}
 %       Let $|\field|\geq4$ and $x\in\field\setminus\{0\}$. Then, the matrix $\begin{bmatrix}
 %           0 & -1\\
 %           1 & x
 %       \end{bmatrix}$ is a commutator of $\mathcal U_2$-matrices if and only if $x=a^2+2$ for some $a\setminus\{0\}$.
 %   \end{lemma}
 %   In 1986, Sourour proved the following factorization result which states that every nonsingular nonscalar matrix is a product of two matrices with prescribed eigenvalues under a determinant condition \cite[Theorem 1]{sourour}.

%\begin{theorem}\cite[Theorem 1]{sourour} \label{thm:sourour}
 %       Let $\field$ be a field. Let $A\in GL_n(\mathbb F)$ be nonscalar and let $b_j,c_j\in\mathbb F$ for all $j\in\{1,...,n\}$ such that $\prod_{j=1}^nb_jc_j=\textup{det}(A)$. Then, there exist $B,C\in GL_n(\mathbb F)$ with $\sigma(B)=\{b_1,...,b_n\}$ and $\sigma(C)=\{c_1,...,c_n\}$ such that $A=BC$.
%\end{theorem}

\begin{lemma} \label{prop:perfsq}
     Let $|\mathbb F|\notin\{2,3,5\}$. There exists $b\in\mathbb F\setminus\{0\}$ such that $b^2\ne b^{-2}$. 
 \end{lemma}
\begin{proof}
 We prove the contrapositive of the claim. Suppose that $x^2=x^{-2}$ for all $x\in\field\setminus\{0\}$. That is, each element of $\field\setminus\{0\}$ is a root of the degree four polynomial $f(x):=x^4-1$. If $\textup{char}(\mathbb F)\neq 2,$ then $|\mathbb F\setminus\{0\}|\leq 4$; in this case, $|\mathbb F|\in \{3,5\}.$ Otherwise, if $\textup{char}(\mathbb F)=2,$ then $f(x)=(x^2+1)^2$ and $|\mathbb F\setminus\{0\}|\leq 2$; in particular, $|\mathbb F|=2.$\end{proof}
 %   We prove the contrapositive of the claim. Suppose that $x^2=x^{-2}$ for all $x\in\field\setminus\{0\}$. That is, each element of $\field\setminus\{0\}$ is a root of $x^4-1$. Since $x^4-1$ has at most four roots in $\field\setminus\{0\}$, it follows that $|\field\setminus\{0\}|\leq 4$. Therefore, $|\field|\in\{2,3,5\}$.

% \begin{proof}
%      Let $S=\{a^2\ |\ a\in\field\setminus\{0\}\}$. We prove the contrapositive of the claim. Suppose that $x^2=x^{-2}$ for all $x\in\field\setminus\{0\}$. That is, each element of $\field\setminus\{0\}$ is a root of $x^4-1$.

%      Let $\chara(\field)=2$. We need to show that $|\field|=2$. Observe that $x^4-1=(x^2+1)^2$ which has at most two distinct roots in $\field\setminus\{0\}$. Thus, $|\field\setminus\{0\}|\leq2$ and therefore, $|\field|\leq3$. Since $\chara(\field)=2$, we have that $|\field|=2$.

%      Let $\chara(\field)\ne2$. We need to show that $|\field|\in\{3,5\}$. Note that $x^4-1$ has at most four roots in $\field\setminus\{0\}$. Thus, $|\field\setminus\{0\}|\leq4$ and therefore, $|\field|\leq5$. Since $\chara(\field)\ne2$ and $\chara(\field)$ is prime, we have that $|\field|\in\{3,5\}$.
% \end{proof}
With the aid of the previous result, we obtain the next one which is analogous to Corollary \ref{thm:J2(-1)} (compare with Corollary \ref{cor:-I_2} as well).

\begin{proposition}\label{-1perfectsquare}
Let $|\mathbb{F}|\notin\{2,3,5\}$. If $-1=a^2$ for some $a\in \mathbb F$, then $-I_2$ is a product of at most two commutators of ${\cal U}_2$-matrices.
\end{proposition}
\begin{proof} By Lemma \ref{prop:perfsq}, there exists $b\in\field\setminus\{0\}$ such that $b^2\ne b^{-2}$. Then $
    -I_2=\textup{diag}(b^2,b^{-2})\ \textup{diag}((ab^{-1})^{2},(ab^{-1})^{-2}) $
    is a product of two commutators of $\mathcal U_2$-matrices by Corollary \ref{lem:hou}.
\end{proof}

Proposition \ref{thm:-I_2} will handle the case when $-1$ is not necessarily a perfect square.
    \begin{proposition} \label{thm:SL2}
        Let $|\mathbb F|\geq4$. Every nonscalar matrix in $SL_2(\mathbb F)$ is a product of at most two commutators of $\mathcal U_2$-matrices.
    \end{proposition}
    \begin{proof}
        Let $|\mathbb F|\geq4$ and $A\in SL_2(\field)$ be a nonscalar matrix. By \cite[Theorem 1]{sourour}, $A=BC$ for some $B,C\in GL_2(\mathbb F)$ where $\sigma(B)$ and $\sigma(C)$ can be specified as long as $\det(B)\det(C)=\det(A)=1$.
        
        If $|\field|\ne5$ (implying $|\mathbb F|\notin\{2,3,5\}$ by assumption), then there exists $b\in\mathbb F\setminus\{0\}$ such that $b^2\ne b^{-2}$ due to Lemma \ref{prop:perfsq}. Take $\sigma(B)=\sigma(C)=\{b^2,b^{-2}\}$. Since $b^2\ne b^{-2}$, $B$ and $C$ are similar to $\diag(b^2,b^{-2})$,
        which is a commutator of $\mathcal U_2$-matrices by Corollary \ref{lem:hou}. Hence, $A$ is a product of two commutators of $\mathcal U_2$-matrices.% by Proposition \ref{prop:rmk2.1}. 

        Suppose that $|\field|=5$. Take $\sigma(B)=\sigma(C)=\{-1\}$. It follows that the minimal polynomials $m_B(x),m_C(x)\in\{x+1,(x+1)^2\}$. Two cases arise: $m_B(x)=m_C(x)$ or not. Suppose $m_B(x)=m_C(x)$. In this case, $m_B(x),m_C(x)$ cannot be both linear since $A$ is a nonscalar matrix. Hence, $m_B(x)=m_C(x)=(x+1)^2$, and so $B$ and $C$ are similar to $J_2(-1)$. Since $-1=2^2$ in $\field$, Proposition \ref{prop:rmk2.1} and Corollary \ref{thm:J2(-1)} ensure that $B,C$ are commutators of $\mathcal U_2$-matrices. Thus, $A$ is a product of two commutators of $\mathcal U_2$-matrices. Suppose $m_B(x)\neq m_C(x)$. Assume $m_B(x)=x+1$ and $m_C(x)=(x+1)^2$ (the case $m_B(x)=(x+1)^2$ and $m_C(x)=x+1$ is analogous). Then there exists $P\in GL_2(\field)$ such that
        \[
        A=BC=(-I_2)PJ_2(-1)P^{-1}=PJ_2(1)^{-1}P^{-1}=(PJ_2(1)P^{-1})^{-1}.
        \]
       Note that $J_2(1)=D^2$ where $D=\begin{bmatrix}
            -1 & 2\\
            0 & -1
        \end{bmatrix}$ is a commutator of $\mathcal U_2$-matrices by Theorem \ref{prop:trace}. Proposition \ref{prop:rmk2.1} guarantees that $A$ is a product of two commutators of $\mathcal U_2$-matrices. 
    \end{proof}

%    If $\chara(\field)=2$, then the only scalar matrix in $SL_2(\field)$ is $I_2$ which is a commutator of $\mathcal U_2$-matrices by Proposition \ref{id}. If $\chara(\field)\ne2$, then the scalar matrices in $SL_2(\field)$ are $I_2$ and $-I_2$. It remains to investigate the decomposition of $-I_2$ as a product of commutators of $\mathcal U_2$-matrices. We prove the following proposition for the matrix $-I_2$ borrowing some techniques from \cite{bien}.

 For some fields, if $-I_2$ is a product of commutators of ${\cal U}_2$-matrices, then the number of factors is at least three (see remarks after Corollary \ref{cor:-I_2}). Such factorization of $-I_2$ into three commutators is indeed possible.

    \begin{proposition} \label{thm:-I_2}
        Let $|\field|\geq 5$ and $\textup{char}(\mathbb F)\neq 2$. Then $-I_2$ is a product of at most three commutators of $\mathcal U_2$-matrices.
    \end{proposition}
    \begin{proof} 
%If $\chara(\mathbb F)=2$, then $-I_2=I_2$ is a commutator of ${\cal U}_2$-matrices by Proposition \ref{id}. Assume $\chara(\mathbb F)\neq 2$. 
%Under the assumption $|\mathbb F|\geq 5$, we consider two cases. 
If $|\field|=5$, then $-I_2=J_2(-1)J_2(1)$ is a product of three commutators of $\mathcal U_2$-matrices due to Corollary \ref{thm:J2(-1)} and Proposition \ref{thm:SL2}. Assume $|\field|>5$. By Lemma \ref{prop:perfsq}, there exists $b\in\mathbb F\setminus\{0\}$ such that $b^2\ne b^{-2}$. In this case, $[-\textup{diag}(b^2,b^{-2})]^{-1}\in SL_2(\mathbb F)$ is a nonscalar matrix. It follows that $-I_2=\textup{diag}(b^2,b^{-2})[-\textup{diag}(b^2,b^{-2})]^{-1}$ is a product of three commutators of ${\cal U}_2$-matrices due to Corollary \ref{lem:hou} and  Proposition \ref{thm:SL2}. 
    \end{proof}

    %    Since $b^2\ne b^{-2}$, Corollary \ref{lem:hou} and Proposition \ref{prop:rmk2.1} imply that first factor is a commutator of $\mathcal U_2$-matrices. On the other hand, the second factor is a product of two commutators of $\mathcal U_2$-matrices by Theorem \ref{thm:SL2}. Hence, $-I_2$ is a product of three commutators of $\mathcal U_2$-matrices.
        
    The main result of this subsection follows from Propositions \ref{id}, \ref{thm:SL2}, \ref{thm:-I_2}, and Corollary \ref{cor:-I_2}.
    
    \begin{theorem} \label{cor:SL2}
        Let $|\field|\geq4$. Every matrix in $SL_2(\mathbb F)$ is a product of at most three commutators of $\mathcal U_2$-matrices. In particular, if $\textup{char}(\mathbb F)=2,$ then every matrix in $SL_2(\mathbb F)$ is a product of at most two commutators of $\mathcal U_2$-matrices; if $\chara(\mathbb F)\neq 2$, then every matrix in $SL_2(\mathbb F)$ is a product of at most two commutators of $\mathcal U_2$-matrices if and only if $-1=a^2+b^2$ for some nonzero $a,b\in \mathbb F$. 
    \end{theorem}
%In particular, three is the least upper bound to the number of commutator of ${\cal U}_2$-matrix factors over all $\mathbb F$ with $|\mathbb F|\geq 4$ 
%%%%%%%%%%%%%%%%%%%%%%%%%%%%%%%%%%%%%%%%%%%%%%%%%%%%%%%%%%%%%%%%%%%%%%
%%%%%%%%%%%%%%%%%%%%%%%%%%%%%%%%%%%%%%%%%%%%%%%%%%%%%%%%%%%%%%%%%%%%%%
%%%%%%%%%%%%%%%%%%%%%%%%%%%SL_n%%%%%%%%%%%%%%%%%%%%%%%%%%%%%%%%%%%%%%%
%%%%%%%%%%%%%%%%%%%%%%%%%%%%%%%%%%%%%%%%%%%%%%%%%%%%%%%%%%%%%%%%%%%%%%
%%%%%%%%%%%%%%%%%%%%%%%%%%%%%%%%%%%%%%%%%%%%%%%%%%%%%%%%%%%%%%%%%%%%%%
 \section{Decomposition of matrices in \texorpdfstring{$SL_n(\field)$}{} into commutators}
\label{sec:SLn}
\setcounter{MaxMatrixCols}{20}
 %We extend the factorization in the previous section to $n\times n$ matrices for $n>2$ 
% For this section, 
In this section, we show that if $n>2$, then every matrix in $SL_n(\mathbb F)$ is a product of at most four commutators of $\mathcal U_2$-matrices provided $|\field|\geq 4$.

As seen from Corollary \ref{thm:J2(1)}, $J_2(1)$ is \textit{not} a commutator of $\mathcal U_2$-matrices. Interestingly, we have the following result.
%$I_n\oplus J_2(1)$ is a commutator of $\mathcal U_2$-matrices. 
\begin{proposition}\label{lem:I+J2(1)}
    The matrix $I_n\oplus J_2(1)$ is a commutator of $\mathcal U_2$-matrices.
\end{proposition}
\begin{proof}
    It suffices to show that $[1]\oplus J_2(1)$ is a commutator by Proposition \ref{prop:I_2+prod}. If $X=\begin{bmatrix}
        1 & 0 & 1\\
        0 & 1 & 0\\
        0 & 0 & 1
    \end{bmatrix}$ and $Y=\begin{bmatrix}
        1 & 0 & 0\\
        -1 & 1 & 0\\
        0 & 0 & 1
    \end{bmatrix}$ (both $\mathcal U_2$-matrices), then $[X,Y]=[1]\oplus J_2(1)$.
\end{proof}

In writing $J_n(1)$ as a product of commutators of ${\cal U}_2$-matrices, we make use of observations about products of matrices with special patterns of zero entries. To emphasize the size of a $k\times \ell $ zero matrix, we denote it by $0_{k,\ell}$ (or as defined earlier, $0_k$ if $k=\ell$). In the next result, $\mathbb F^n$ denotes the set of all $n\times 1$ vectors with entries in $\mathbb F.$

%Sometimes the subscripts are omitted; in these cases, context determines the size.

%\begin{lemma}\label{lem:zero}
%Let $k,n\in\mathbb N$ such that $k\leq n-2$. Suppose $R,S\in M_n(\mathbb F)$ such that $R=\begin{bmatrix} 0_k& R_{12}\\ 0& R_{22}\end{bmatrix}$, $S=\begin{bmatrix} S_{11}& S_{12}\\ 0 &S_{22}\end{bmatrix}$, and the first two columns of  $S_{22}\in M_{n-k}(\mathbb F)$ are zero. If $k=n-2$, then $RS=0_n$; if $k<n-2$, then $RS=\begin{bmatrix} 0_{k+2}& T_{12}\\ 0 &T_{22}\end{bmatrix}$.
%\end{lemma}

\begin{lemma}\label{lem:zero}
Given $k\geq 2$, $\ell\geq 1$, $x\in \mathbb F^2$, $B,C\in M_2(\mathbb F)$, let $A_e:=[A_{ij}]\in M_{2k}(\mathbb F)$ be strictly block upper triangular matrix where $A_{ij}=0_2$ whenever $j\geq i+3$, $A_{i,i+1}=B$ for all $i\in\{1,\ldots,k-1\}$, and if $k>2$, $A_{i,i+2}=C$ for all $i\in\{1,\ldots,k-2\}$; let $A_o:=\begin{bmatrix} 0_2 & Bx\\ 0_{1,2}& 0\end{bmatrix}$ if $\ell=1 $ and $A_o:=\begin{bmatrix} A_e& N\\ 0_{1,2\ell}& 0\end{bmatrix}$ if $\ell>1$ where $ N\in \mathbb{F}^{2l}$ is zero except possibly in its last four rows which are given by the rows of $\begin{bmatrix} Cx\\ Bx\end{bmatrix}$. The following hold:
\begin{enumerate}[(i)]
    \item $A_e^k=0_{2k}$ and $A_e^{k-1}=\begin{bmatrix}0_{2,2k-2}& B^{k-1}\\ 0_{2k-2}& 0_{2k-2,2}\end{bmatrix}$;
    \item $A_o^{\ell+1}=0_{2\ell+1}$ and $A_o^{\ell}=\begin{bmatrix} 0_{2\ell} & B^\ell x\\ 0_{1,2\ell} & 0\end{bmatrix}.$
\end{enumerate}
\end{lemma}

%such that $k\leq n-2$. Suppose $R,S\in M_n(\mathbb F)$ such that $R=\begin{bmatrix} 0_k& R_{12}\\ 0& R_{22}\end{bmatrix}$, $S=\begin{bmatrix} S_{11}& S_{12}\\ 0 &S_{22}\end{bmatrix}$, and the first two columns of  $S_{22}\in M_{n-k}(\mathbb F)$ are zero. If $k=n-2$, then $RS=0_n$; if $k<n-2$, then $RS=\begin{bmatrix} 0_{k+2}& T_{12}\\ 0 &T_{22}\end{bmatrix}$.
\begin{proof}
The claims in (i) follow by repeated multiplication and recognizing that 
\[A_e=\begin{bmatrix} 0_2& B \\ 0_2& 0_2\end{bmatrix}\ \textup{if}\ k=2\ \textup{and}\ A_e=\begin{bmatrix}0_2& B& C& 0_2&\cdots& 0_2\\ 0_2& 0_2& B& C& \ddots & \vdots\\ 0_2& 0_2& 0_2& B&\ddots& 0_2 \\  0_2& 0_2& 0_2& 0_2&\ddots& C\\ \vdots& \ddots& \ddots& \ddots&\ddots& B\\  0_2& \cdots& 0_2& 0_2& 0_2& 0_2\\ \end{bmatrix}\ \textup{if}\ k>2.\] For (ii), the case $\ell=1$ is easily verified; if $\ell>1$, then use (i) and the fact that $A_o^i=\begin{bmatrix}A_e^i& A_e^{i-1}N\\ 0_{1,2\ell}& 0\end{bmatrix}$ for any $i$.

%Let $S_{22}=\begin{bmatrix} 0_{n-k,2}& Z\end{bmatrix}$ where $Z $ is absent if $k=n-2$. The desired form can be seen from 
%\[RS=\begin{bmatrix} 0_k& R_{12}S_{22}\\ 0 & R_{22}S_{22}\end{bmatrix}=\begin{bmatrix} 0_k& R_{12}\begin{bmatrix} 0_{n-k,2}& Z\end{bmatrix}\\ 0 & R_{22}\begin{bmatrix} 0_{n-k,2}& Z\end{bmatrix}\end{bmatrix}=\begin{bmatrix} 0_k& \begin{bmatrix} 0_{n-k,2}& R_{12}Z\end{bmatrix}\\ 0 & \begin{bmatrix} 0_{n-k,2}& R_{22}Z\end{bmatrix}\end{bmatrix}.\]
\end{proof}

%If  R_{12} and Z is "upper triangular" with nonzero "diagonal" entries, then RS\ne 0? 

%We now show that $J_n(1)$ is a product of at most two commutators of $\mathcal U_2$-matrices.
\begin{proposition}\label{lem:Jn(1)}
    Let $n>2$. The matrix $J_n(1)$ is a product of at most two commutators of $\mathcal U_2$-matrices.   
\end{proposition}
\begin{proof}
 %  If $n=2$, then $J_2(1)$ is a product of at most two commutators of $\mathcal U_2$-matrices due to Theorem \ref{thm:SL2}. 
    
   % Assume $n>2$.
    Observe that
    \[J_n(1)\ \textup{is similar to}\ \begin{cases}[J_\frac{n}{2}(1)\oplus J_\frac{n}{2}(1)][I_\frac{n-2}{2}\oplus J_2(1)\oplus I_\frac{n-2}{2}],&\textup{if}\ n\ \textup{is even}\\ [J_\frac{n+1}{2}(1)\oplus J_\frac{n-1}{2}(1)][I_\frac{n-1}{2}\oplus J_2(1)\oplus I_\frac{n-3}{2}],& \textup{if}\ n\ \textup{is odd}.\end{cases}\] By Propositions \ref{prop:rmk2.1}-\ref{prop:I_2+prod}, and \ref{lem:I+J2(1)}, it suffices to prove that the first factor above is similar to a commutator of ${\cal U}_2$-matrices.
    
    Consider \[\begin{cases}\vspace{0.2cm} X_n:=[1]\oplus\left(\bigoplus\limits_{i=1}^\frac{n-2}{2}J_2(1)\right)\oplus[1]\ \textup{and}\ Y_n:=\bigoplus\limits_{i=1}^\frac{n}{2}J_2(1),& \textup{if}\ n\ \textup{is even}\\ X_n:=[1]\oplus\left(\bigoplus\limits_{i=1}^\frac{n-1}{2}J_2(1)\right)\ \textup{and}\ Y_n:=\left(\bigoplus\limits_{i=1}^\frac{n-1}{2}J_2(1)\right)\oplus[1],& \textup{if}\ n\ \textup{is odd}\end{cases}\] which are $\mathcal U_2$-matrices. %Direct computations reveal that 
    % Then, 
    % \[
    % AB=\begin{bmatrix}
    %     1 & 1 &   &   &   &   &        &   &   &   &  \\
    %       & 1 & 1 & 1 &   &   &        &   &   &   &  \\
    %       &   & 1 & 1 &   &   &        &   &   &   &  \\
    %       &   &   & 1 & 1 & 1 &        &   &   &   &  \\
    %       &   &   &   &   &   &        &   &   &   &  \\
    %       &   &   &   &   &   & \ddots &   &   &   &  \\
    %       &   &   &   &   &   &        &   &   &   &  \\
    %       &   &   &   &   &   &        & 1 & 1 &   &  \\
    %       &   &   &   &   &   &        &   & 1 & 1 & 1\\
    %       &   &   &   &   &   &        &   &   & 1 & 1\\
    %       &   &   &   &   &   &        &   &   &   & 1\\
    % \end{bmatrix}
    % \]
    % and
    % \[
    % A^{-1}B^{-1}=\begin{bmatrix}
    %     1 & -1 &    &    &    &   &        &   &    &    &   \\
    %       & 1  & -1 & 1  &    &   &        &   &    &    &   \\
    %       &    & 1  & -1 &    &   &        &   &    &    &   \\
    %       &    &    & 1  & -1 & 1 &        &   &    &    &   \\
    %       &    &    &    &    &   &        &   &    &    &   \\
    %       &    &    &    &    &   & \ddots &   &    &    &   \\
    %       &    &    &    &    &   &        &   &    &    &   \\
    %       &    &    &    &    &   &        & 1 & -1 &    &   \\
    %       &    &    &    &    &   &        &   & 1  & -1 & 1 \\
    %       &    &    &    &    &   &        &   &    & 1  & -1\\
    %       &    &    &    &    &   &        &   &    &    & 1 \\
    % \end{bmatrix}.
    % \]
    Let $A_n:=[X_n,Y_n]-I_n$, $B:=\begin{bmatrix} -1& 1\\ 0& 1\end{bmatrix}$, $C:=\begin{bmatrix}0& 0\\ -1& 1\end{bmatrix}$, and $x:=\begin{bmatrix}1\\ 0\end{bmatrix}.$ If $n=2k$ where $k\geq 2$ (since $n>2$), then $A_{2k}$ has the form of $A_e$ in Lemma \ref{lem:zero}
    %\begin{equation}\label{evenA}
    %A_{2k}=
    %\begin{bmatrix}0_2& B& C& 0_2&\cdots& 0_2\\ 0_2& 0_2& B& C& \ddots & \vdots\\ 0_2& 0_2& 0_2& B&\ddots& 0_2 \\  0_2& 0_2& 0_2& 0_2&\ddots& C\\ \vdots& \ddots& \ddots& \ddots&\ddots& B\\  0_2& \cdots& 0_2& 0_2& 0_2& 0_2\\ \end{bmatrix}
    %\end{equation} 
    while if $n=2\ell+1$ where $\ell\geq 1$ (since $n>2$), then $A_{2\ell+1}$ has the form of $A_o$ in Lemma \ref{lem:zero}. Due to the strictly block upper triangular structure of $A_n$ and nonsingularity of $B$, $\sigma([X_n,Y_n])=\{1\}$, $\textup{nullity}([X_n,Y_n]-I_n)=2$, and hence, $[X_n,Y_n]$ only has two Jordan blocks corresponding to 1. Moreover, Lemma \ref{lem:zero} yields $A_n^{\lceil \frac{n}{2}\rceil}=0_n$ and $A_n^{\lceil \frac{n}{2}\rceil-1}\neq 0_n$
    %\[([X_n,Y_n]-I_n)^{\lceil \frac{n}{2}\rceil}=0_n\ \textup{and}\ ([X_n,Y_n]-I_n)^{\lceil \frac{n}{2}\rceil-1}\neq 0_n\] 
    since $B$ is nonsingular (here, $\lceil x\rceil$ denotes the ceiling function of $x\in\mathbb R$). Thus,
    \[[X_n,Y_n]\ \textup{is similar to}\ \begin{cases} J_{\frac{n}{2}}(1)\oplus J_{\frac{n}{2}}(1),& \textup{if}\ n\ \textup{is even}\\ J_\frac{n+1}{2}(1)\oplus J_\frac{n-1}{2}(1),& \textup{if}\ n\ \textup{is odd}.\end{cases}\] This completes the proof of the claim.\end{proof}

\begin{proposition} \label{mainnonscalar}
    Let $n>2$ and $|\field|\geq4$. Every nonscalar matrix in $SL_n(\mathbb F)$ is a product of at most four commutators of $\mathcal U_2$-matrices.
\end{proposition}
\begin{proof}
   Let $A\in SL_n(\mathbb F)$ be a nonscalar matrix. By \cite[Theorem 1]{sourour}, there exist $B,C\in GL_n(\field)$ with $\sigma(B)=\sigma(C)=\{1\}$ such that $A=BC$. Hence, $B$ and $C$ are respectively similar to $\bigoplus\limits_{i=1}^rJ_{m_i}(1)$ and $\bigoplus\limits_{j=1}^sJ_{n_j}(1)$ for some $r,s\in \mathbb N$ where $1\leq m_i,n_j\leq n$ for all $i,j$. Both $B$ and $C$ are products of at most two commutators of ${\cal U}_2$-matrices due to Propositions \ref{prop:rmk2.1}-\ref{prop:I_2+prod}, \ref{thm:SL2}, and \ref{lem:Jn(1)}. Thus, $A$ is a product of at most four commutators of $\mathcal U_2$-matrices.
\end{proof}

In the above proof, even if $J_2(1)$ occurs as a direct summand of either $B$ or $C$, then $J_2(1)$ is a product of two commutators of ${\cal U}_2$-matrices due to the assumption that $|\mathbb F|\geq 4$ and Proposition \ref{thm:SL2}.

Next, we consider scalar matrices in $SL_n(\mathbb F)$. As seen from the proofs of Propositions \ref{thm:SL2}-\ref{thm:-I_2}, the field with $|\mathbb F|=5$ demands special attention. %In this case, $-1=a^2+b^2$ has no \textit{nonzero} solution in $\mathbb F$, and so $-I_2$ is a product of \textit{exactly three} commutators of ${\cal U}_2$-matrices due to 
%Propositions \ref{id}, \ref{thm:-I_2}, and Corollary \ref{cor:-I_2}. However, once $n>2 $ is assumed, $-I_n$ can now be written as a product of \textit{two} commutators of ${\cal U}_2$-matrices.

\begin{lemma}\label{F5}
Let $n\geq2$ be even and $|\mathbb F|=5.$ If $A\in SL_n(\mathbb F)$ is a scalar matrix, then $A$ is a product of at most four commutators of ${\cal U}_2$-matrices.
\end{lemma}
\begin{proof}
 Let $A\in SL_n(\field)$ such that $A=\lambda I_n$ for some $\lambda\in\field$. Note that $\lambda^n=\det(A)=1$. If $A=I_n$, then $A$ is a commutator of ${\cal U}_2$-matrices due to Proposition \ref{id}. If $A=4I_n$, then $A=-I_n=\bigoplus\limits_{i=1}^k -I_2$ since $\textup{char} (\mathbb F)=5$ and $n$ is even. By Propositions \ref{prop:I_2+prod} and \ref{thm:-I_2}, $A$ is a product of at most three commutators of ${\cal U}_2$-matrices. For the remaining cases $2I_n$ and $3I_n$, it suffices to prove the claim for $A=2I_n$ due to Proposition \ref{prop:rmk2.1} and $3I_n=(2I_n)^{-1}$. Since $n$ is even, $1=2^{n}=4^{\frac{n}{2}}=(-1)^{\frac{n}{2}}$. Hence, $n=4k$ for some $k\in \mathbb N$ since $\textup{char}(\mathbb F)=5$, and so $A=\bigoplus\limits_{i=1}^k 2I_4$. By Proposition \ref{prop:I_2+prod}, the claim about $A$ follows by showing that $2I_4$ is a product of at most four commutators of ${\cal U}_2$-matrices. Note that $2I_4=(B\oplus B)C$ where $B=\textup{diag}(2,3)$ and $C=\textup{diag}(1,-1,1,-1)$. Since $B\in SL_2(\mathbb F)$ is a nonscalar matrix, $B\oplus B$ is a product of at most two commutators of ${\cal U}_2$-matrices by Propositions \ref{prop:I_2+prod} and \ref{thm:SL2}. On the other hand, let $D=[1]\oplus \begin{bmatrix} -1& 0 & 1\\ 0 & 1& 0\\ 0 & 0& -1\end{bmatrix}$ and $E=[1]\oplus J_2(-1)\oplus[1]$ (both commutators of ${\cal U}_2$-matrices by Proposition \ref{prop:I_2+prod} and Corollary \ref{thm:J2(-1)}), and observe that \[DE=[1]\oplus\begin{bmatrix}1 & -1& 1\\ 0&-1&0\\ 0 & 0& -1\end{bmatrix}\]
 is similar to $C$. The claim follows.
\end{proof}

\begin{lemma}\label{prop:finperf}
Let $S=\{a^2|a\in\field\setminus\{0\}\}$. If $|\field|$ is even, then $|S|=|\field|-1$. If $|\field|$ is odd, then $|S|=\frac{|\field|-1}{2}$.
\end{lemma}
\begin{proof}
    The map $\phi:\field\setminus\{0\}\to\field\setminus\{0\}$ defined by $\phi(x)=x^2$ is a group homomorphism. By the First Isomorphism Theorem for groups, $\field\setminus\{0\}/\ker(\phi)\cong\textup{Im}(\phi)=S.$
    If $|\field|$ is even, then $\chara(\field)=2$ and so $\ker(\phi)=\{1\}$. Consequently, $|S|=|\field|-1$. On the other hand, $|\field|$ being odd implies $\chara(\field)\ne2$. In particular, $|\ker(\phi)|=|\{-1,1\}|=2$, and so $|S|=\frac{|\field|-1}{2}$.
\end{proof}

\begin{proposition} \label{mainscalar}
    Let $n\geq 2$ and $|\field|\geq4$. Let $A\in SL_n(\field)$ be a scalar matrix.
    \begin{enumerate}[(i)]
        \item If $n$ is odd, then $A$ is a product of at most two commutators of $\mathcal U_2$-matrices;
        \item Let $n$ be even. Then $A$ is a product of at most four commutators of $\mathcal U_2$-matrices. In addition, if $\chara(\field)=2$, then $A$ is product of at most two commutators of $\mathcal U_2$-matrices.
    \end{enumerate}
\end{proposition}
\begin{proof}
    Let $A\in SL_n(\field)$ such that $A=\lambda I_n$ for some $\lambda\in\field$. For each $i\in \{0,1,\ldots,n\}$, define $\Lambda_i:=\textup{diag}(\lambda^i,\lambda^{n-i})$. Since $\lambda^n=\det(A)=1$, $\Lambda_i\in SL_2(\mathbb F)$ and $\Lambda_i\Lambda_{n-i}=I_2$.
    
    Suppose $n=2k+1$ for some $k\in\mathbb N$. Let $b:=\lambda^{\frac{n+1}{2}}\in\field$. 
    For each $i$, $\lambda^i=(\lambda^n\lambda)^i=(\lambda^{\frac{n+1}{2}})^{2i}=b^{2i}=(b^i)^2.$ Proposition \ref{id} or Corollary \ref{lem:hou} implies that $\Lambda_i$ is a commutator of ${\cal U}_2$-matrices except possibly when $\lambda^i=-1$. However, in such a case, $1=(\lambda^n)^i=(\lambda^i)^n=(-1)^n=-1$, and so $\Lambda_i=I_2$ which is a commutator of ${\cal U}_2$-matrices by Proposition \ref{id}. In any case, $\Lambda_i$ is a commutator of ${\cal U}_2$-matrices. Now, note that
    \[A= \left[\left(\bigoplus\limits_{i=1}^k \Lambda_i \right)\oplus [1]\right] \left[\left(\bigoplus\limits_{i=1}^k\lambda\Lambda_{n-i}\right)\oplus [\lambda]\right].\]
Moreover, the second factor above is permutation similar to the first factor.
   By the preceding remarks and Propositions \ref{id}-\ref{prop:I_2+prod}, $A$ is a product of at most two commutators of $\mathcal U_2$-matrices.
   
% \color{red} $|\field|\ne5$??\color{black}
 
    Suppose $n=2k$ for some $k\in\mathbb N$. If $|\mathbb F|=5$, then the claim follows from Lemma \ref{F5}. Assume $|\mathbb F|\neq 5$. Note that
    \[
    A= \left(\bigoplus\limits_{i=1}^k \Lambda_{2i-1}  \right)\left(\bigoplus\limits_{i=1}^k\lambda\Lambda_{n-(2i-1)}\right).\]
Let $B$ and $C$ be respectively the first and second factors above. Observe that $C$ is permutation similar to
    \begin{equation} \label{scalarevenreducsecond}
        I_2\oplus\left(\bigoplus\limits_{i=1}^{k-1}\Lambda_{2i}\right).
    \end{equation}
Moreover, the diagonal entries of $\Lambda_{2i}$ are perfect squares. We consider cases depending on $\textup{char}(\mathbb F)$. If $\textup{char}(\mathbb F)=2$, then the diagonal entries (including possibly $-1=1$) of $\Lambda_{2i-1}$ are perfect squares due to Lemma \ref{prop:finperf}. Hence, $A$ is a product of at most two commutators of ${\cal U}_2$-matrices (namely, $B$ and $C$) due to Propositions \ref{id}-\ref{prop:I_2+prod} and Corollary \ref{lem:hou}. Assume $\textup{char}(\mathbb F)\neq 2.$ Suppose that in \eqref{scalarevenreducsecond}, $\Lambda_{2i}\neq -I_2$ for all $i\in\{1,\ldots,k-1\}.$ Then $C$ is a commutator of ${\cal U}_2$-matrices by Propositions \ref{id}-\ref{prop:I_2+prod} and Corollary \ref{lem:hou}. On the other hand, $B$ is a product of at most three commutators of ${\cal U}_2$-matrices due to Propositions \ref{id}-\ref{prop:I_2+prod} and Theorem \ref{cor:SL2}. The claim about $A$ follows in this case. For the remaining case, suppose that in \eqref{scalarevenreducsecond}, $\Lambda_{2i}=-I_2$ for some $i\in \{1,\ldots, k-1\}$. In particular, $-1=a^2$ where $a=\lambda^i$. By Proposition \ref{-1perfectsquare}, all occurrences of $-I_2$ in $B$ and in \eqref{scalarevenreducsecond} are products of two commutators of ${\cal U}_2$-matrices. The remaining direct summands in $B$ and in \eqref{scalarevenreducsecond} are either $I_2$ or nonscalar matrices in $SL_2(\mathbb F).$ Hence, both $B$ and $C$ are products of at most two commutators of ${\cal U}_2$-matrices due to Propositions \ref{id}-\ref{prop:I_2+prod} and \ref{thm:SL2}. This proves the claim about $A$.\end{proof}

 Combining Propositions \ref{mainnonscalar} and \ref{mainscalar}, we now have the main result of this section and a generalization of \cite[Theorem 1.1]{hou}.
 
\begin{theorem} \label{mainthm}
    Let $n>2$ and $|\mathbb F|\geq 4$. Every matrix in $SL_n(\field)$ is a product of at most four commutators of $\mathcal U_2$-matrices.
\end{theorem}

Proposition \ref{rmk:houtowang} and Theorems \ref{cor:SL2} and \ref{mainthm} imply an analogue of \cite[Theorem 3.5]{wangwu}.

\begin{corollary}\label{cor:wangwu}
    Let $|\mathbb F|\geq 4$. Every matrix in $SL_2(\mathbb F)$ is a product of at most six ${\cal U}_2$-matrices. If $n>2$, then every matrix in $SL_n(\mathbb F)$ is a product of at most eight ${\cal U}_2$-matrices.
\end{corollary}

%%%%%%%%%%%%%%%%%%%%%%%%%%%%%%%%%%%%%%%%%%%%%%%%%%%%%%%%%%%%%%%%%%%%%%
%%%%%%%%%%%%%%%%%%%%%%%%%%%%%%%%%%%%%%%%%%%%%%%%%%%%%%%%%%%%%%%%%%%%%%
%%%%%%%%%%%%%%%%%%%%%%%%%%%REDUCTION%%%%%%%%%%%%%%%%%%%%%%%%%%%%%%%%%%
%%%%%%%%%%%%%%%%%%%%%%%%%%%%%%%%%%%%%%%%%%%%%%%%%%%%%%%%%%%%%%%%%%%%%%
%%%%%%%%%%%%%%%%%%%%%%%%%%%%%%%%%%%%%%%%%%%%%%%%%%%%%%%%%%%%%%%%%%%%%%
\section{Reduction of commutator factors}\label{sec:reduc}

In this section, we consider sufficient conditions on  $\mathbb F$ (e.g., conditions about $|\mathbb F|$ or $\chara(\field)$) that improve the upper bound on the commutator factors in Theorem \ref{mainthm}.

\begin{lemma} \label{prop:pairsfin}
    Let $|\field|\notin\{2,3,5\}$ and consider $S=\{a^2| a\in\field\setminus\{0\}\}$. Then $S$ can be written as a disjoint union of sets in the following manner:
    \begin{equation}\label{eq:pair}
        S=E\cup\{\alpha_1,\alpha_1^{-1},\alpha_2,\alpha_2^{-1},...\}
    \end{equation}
    where $\alpha_1,\alpha_1^{-1},\alpha_2,\alpha_2^{-1},...$ are distinct and $E$ is defined as
    \begin{center}
        $E=\begin{cases}
            \{1\} & \textup{if $\chara(\field)=2$}\\
            \{1\} & \textup{if $\chara(\field)\ne2$ and $-1\notin S$}\\
            \{-1,1\} & \textup{if $\chara(\field)\ne2$ and $-1\in S$}.
        \end{cases}$
    \end{center}
    If $\field$ is infinite, then $S\setminus E$ is infinite. If $\field$ is finite, then $S\setminus E$ is finite and
    \[
    S\setminus E=\{\alpha_1,\alpha_1^{-1},...,\alpha_k,\alpha_k^{-1}\}
    \]
    for some $k\in\mathbb N$ where
    \begin{enumerate}[(i)]
                \item $k=\frac{|\field|-2}{2}$ if $\chara(\field)=2$;
                \item $k=\frac{|\field|-3}{4}$ if $\chara(\field)\ne2$ and $-1\notin S$, and
                \item $k=\frac{|\field|-5}{4}$ if $\chara(\field)\ne2$ and $-1\in S$.
            \end{enumerate}
\end{lemma}
\begin{proof}
  Since $|\field|\notin\{2,3,5\}$, Lemma \ref{prop:perfsq} implies that there exists $\beta\in S$ such that $\beta\ne\beta^{-1}$. In particular, $\beta\notin\{1,-1\}$, and so $\beta\notin E$. This guarantees that $S\setminus E\ne\varnothing$. Both $\alpha,\alpha^{-1}\in S$ since $S$ is a multiplicative group. We claim that $\alpha\in S\setminus E$ if and only if $\alpha^{-1}\in S\setminus E.$ Indeed, suppose $\alpha \in S\setminus E$. If, on the contrary, $\alpha^{-1}\in E,$ then we arrive at a contradiction in each possible case of $E$ as defined; hence, $\alpha^{-1}\in S\setminus E$. The converse is analogous. Moreover, if $\alpha\in S\setminus E$, then $\alpha\neq \alpha^{-1}.$ Otherwise, $\alpha\in\{-1,1\}$; checking each case of $E$ yields a contradiction. Therefore, we have the disjoint union $S=E\cup (S\setminus E)$, where $\alpha\in S\setminus E$ if and only if $\alpha^{-1}\in S\setminus E$ and $\alpha\ne\alpha^{-1}$.

   % Let $\alpha\in T=S\setminus E$. Then $\alpha\ne\pm1$, and so $\alpha\neq \alpha^{-1}$. Since $S$ is a multiplicative group, $\alpha^{-1}\in S$. Note that $\alpha^{-1}\in T$. Conversely, if $\alpha\in S$ such that $\alpha\ne\alpha^{-1}$ and $\alpha^{-1}\in T$, then $\alpha^2\ne\pm1$ and thus, $\alpha\ne\pm1$. Consequently, it follows that $\alpha\in T$. Therefore, $S=E\cup T$, where $\alpha\in T$ if and only if $\alpha^{-1}\in T$ and $\alpha\ne\alpha^{-1}$.

    If $\field$ is infinite, then $S$ and $S\setminus E$ are infinite. Assume $\mathbb F$ is finite. By Lemma \ref{prop:finperf}, $S$ is also a finite set. There exists $k\in\mathbb N$ such that 
    \begin{equation} \label{eq:finpair}
        S=E\cup\{\alpha_1,\alpha_1^{-1},...,\alpha_k,\alpha_k^{-1}\}
    \end{equation}
    where $E\cap\{\alpha_1,\alpha_1^{-1},...,\alpha_k,\alpha_k^{-1}\}=\varnothing$ and $\alpha_1,\alpha_1^{-1},...,\alpha_k,\alpha_k^{-1}$ are distinct. If $|\mathbb F|$ is even, then by definition of $E$ and Lemma \ref{prop:finperf}, $|\mathbb F|-1=|S|=1+2k$. Hence, $k=\frac{|\mathbb{F}|-2}{2}$. If $|\mathbb{F}|$ is odd, then $\frac{|\mathbb{F}|-1}{2}=|S|=|E|+2k$ by Lemma \ref{prop:finperf}, and so $k=\frac{|\mathbb F|-1-2|E|}{4}$. The claim follows by considering cases for $E.$\end{proof}

Let $\lfloor x\rfloor$ denote the floor function of $x\in\mathbb R$.
\begin{proposition} \label{mainlem2} 
        Every nonscalar matrix in $SL_n(\mathbb F)$ is a product of at most two commutators of $\mathcal U_2$-matrices provided that any of the following conditions hold:
        \begin{enumerate}[(i)]
            \item $|\field|\geq2\lfloor\frac{n}{2}\rfloor+2$ if $\chara(\field)=2$; 
            \item $|\field|\geq4\lfloor\frac{n}{2}\rfloor+3$ if $\chara(\field)\ne2$ and $-1\ne \alpha^2$ for all $\alpha\in\field$;
            \item $|\field|\geq4\lfloor\frac{n}{2}\rfloor+5$ if $\chara(\field)\ne2$ and $-1=\alpha^2$ for some $\alpha\in\field$.
        \end{enumerate}
        In particular, if $|\field|\geq4\lfloor\frac{n}{2}\rfloor+5$, then every nonscalar matrix in $SL_n(\mathbb F)$ is a product of at most two commutators of ${\cal U}_2$-matrices.
    \end{proposition}
    \begin{proof}        
        Suppose $n=2k$ or $n=2k+1$ for some $k\in\mathbb N$. Let $S$ and $E$ be as in Proposition \ref{prop:pairsfin}. If $\mathbb F$ is infinite, then $S\setminus E$ is infinite, and so $2k\leq |S\setminus E|$. If $\mathbb F$ is finite, then assumptions (i)-(iii) guarantee that $2k\leq |S\setminus E|$. In any case, there exist distinct $\alpha_1,\alpha_1^{-1}...,\alpha_k,\alpha_k^{-1}\in S\setminus E$. By \cite[Theorem 1]{sourour}, $A=BC$ for some $B$ and $C$ where $\sigma(B)$ and $\sigma(C)$ can be specified provided that $\det(B)\det(C)=\det(A)=1$. 
    If $n=2k$, let 
    $\sigma(B)=\sigma(C)=\{\alpha_1,\alpha_1^{-1},...,\alpha_k,\alpha_k^{-1}\}$ while if $n=2k+1$, let $\sigma(B)=\sigma(C)=\{1,\alpha_1,\alpha_1^{-1},...,\alpha_k,\alpha_k^{-1}\}$. Since the eigenvalues of $B$ and $C$ are distinct, both matrices are similar to $\diag(\alpha_1,\alpha_1^{-1},...,\alpha_k,\alpha_k^{-1})$ (if $n=2k$) or $\diag(1,\alpha_1,\alpha_1^{-1},...,\alpha_k,\alpha_k^{-1})$ (if $n=2k+1$). For all $i\in\{1,...,k\}$, note that $\alpha_i\in S$ and so by Corollary $\ref{lem:hou}$, $\diag(\alpha_i,\alpha_i^{-1})$ is a commutator of $\mathcal U_2$-matrices. It follows that $B$ and $C$ are commutators of $\mathcal U_2$-matrices by Propositions \ref{prop:rmk2.1}-\ref{prop:I_2+prod}. Hence, $A$ is a product of at most two commutators of $\mathcal U_2$-matrices.   \end{proof}
     %   Suppose $n=2k$ or $n=2k+1$ for some $k\in\mathbb N$. Let $S$ and $E$ be as in Proposition \ref{prop:pairsfin}. If $\chara(\field)=2$, then $k=\lfloor\frac{n}{2}\rfloor\leq\frac{q-2}{2}$ by (a). If $\chara(\field)\ne2$ and $-1\notin S$, then $k=\lfloor\frac{n}{2}\rfloor\leq\frac{q-3}{4}$ by (b). If instead we have $\chara(\field)\ne2$ and $-1\in S$, we then get $k=\lfloor\frac{n}{2}\rfloor\leq\frac{q-5}{4}$ by (c). In each of the cases, we can choose distinct $\alpha_1,\alpha_1^{-1},...,\alpha_k,\alpha_k^{-1}\in S\setminus E$ by Proposition \ref{prop:pairsfin}. By \cite[Theorem 1]{sourour}, we have $A=BC$ for some $B,C\in GL_n(\field)$ with $\sigma(B)=\sigma(C)=\{\alpha_1,\alpha_1^{-1},...,\alpha_k,\alpha_k^{-1}\}$ if $n=2k$ while $\sigma(B)=\sigma(C)=\{1,\alpha_1,\alpha_1^{-1},...,\alpha_k,\alpha_k^{-1}\}$ if $n=2k+1$. Note that $\diag(\alpha_i,\alpha_i^{-1})$ is a commutator of $\mathcal{U}_2$ matrices for all $i\in\{1,...,k\}$ due to Corollary \ref{lem:hou}. Since the eigenvalues of $B$ and $C$ are distinct, both are similar to $\diag(\alpha_1,\alpha_1^{-1},...,\alpha_k,\alpha_k^{-1})$ (if $n=2k$) or $\diag(1,\alpha_1,\alpha_1^{-1},...,\alpha_k,\alpha_k^{-1})$ (if $n=2k+1$). It follows that $B$ and $C$ are commutators of $\mathcal U_2$-matrices by Propositions \ref{prop:rmk2.1}-\ref{prop:I_2+prod}. Consequently, $A$ is a product of two commutators of $\mathcal U_2$-matrices.
  
    %We also provide additional results on scalar matrices in $SL_n(\field)$. 
   % For odd $n$, . 
    The next result improves upon the bound given by Proposition \ref{mainscalar} (ii). %We only consider even $n$ since for odd $n$, every scalar matrix in $SL_n(\field)$ is a product of at most two commutators of $\mathcal{U}_2$-matrices.
    
\begin{proposition} \label{lem:scalars}
    Let $n$ be even and $|\field|>2n+1$. Every scalar matrix in $SL_n(\mathbb F)$ is a product of at most three commutators of $\mathcal U_2$-matrices.
           In addition, if $\field$ is algebraically closed, then every scalar matrix in $SL_n(\mathbb F)$ is a product of at most two commutators of $\mathcal U_2$-matrices. 
\end{proposition}
\begin{proof}
    Let $A\in SL_n(\field)$ such that $A=\lambda I_n$ for some $\lambda\in\field$. Define $\Lambda_i:=\diag(\lambda^i,\lambda^{n-i})$ for $i\in\{0,1,...,n-1\}$. Since $\lambda^n=\det(A)=1$, $\Lambda_i\in SL_n(\mathbb F)$ and $\Lambda_i\Lambda_{n-i}=I_2.$ Assume $n=2k$ for some $k\in\mathbb N$.
    
    Since $|\field\setminus\{0\}|=|\mathbb F|-1>2n$, there exists $a\in\field\setminus\{0\}$ such that $a^{2n}\ne1$. Set $D:=\textup{diag}(a^2,a^{-2}).$ Note that
    \[A=\left(\bigoplus\limits_{i=1}^k\lambda^{2i-1}D\right)\left(\bigoplus\limits_{i=1}^k\lambda^{2-2i}D^{-1}\right).\]
    Let $B$ and $C$ be respectively the first and second factors above. By permuting diagonal entries, $B$ is similar to $\bigoplus\limits_{i=1}^k\Lambda_{2i-1}D$ while $C$ is similar to $\bigoplus\limits_{i=1}^k\Lambda_{2i-2}^{-1}D^{-1}.$ Now, observe that for all $i\in\{1,\ldots,k\}$, $\lambda^{2-2i}a^{-2}\ne-1$ and $\lambda^{2i-1}a^2\ne-1$. Otherwise, since $n$ is even, either $a^{2n}=(-\lambda^{2-2i})^n=(\lambda^n)^{2-2i}=1$ or $a^{2n}=(-\lambda^{1-2i})^n=(\lambda^n)^{1-2i}=1$, contradicting the choice of $a.$ Hence, for all $i\in\{1,...,k\}$, $\Lambda_{2i-2}^{-1}D^{-1}$ is a commutator of $\mathcal U_2$-matrices by Proposition \ref{id} or Corollary \ref{lem:hou}. By Propositions \ref{prop:rmk2.1}-\ref{prop:I_2+prod}, $C$ is a commutator of $\mathcal U_2$-matrices. On the other hand, for all $i\in \{1,\ldots,k\}$, $\Lambda_{2i-1}D\in SL_2(\field)\setminus\{-I_2\}$. By Propositions \ref{id}-\ref{prop:I_2+prod} and \ref{thm:SL2}, $B$ is a product of at most two commutators of $\mathcal U_2$-matrices. Hence, $A$ is a product of at most three commutators of ${\cal U}_2$-matrices.
    
    Suppose $\field$ is algebraically closed. For each $i\in\{1,\ldots,k\}$, there exists $r_i\in\mathbb F$ such that $r_i^2=\lambda^{2i-1}$. By Proposition \ref{id} or Corollary \ref{lem:hou}, $\Lambda_{2i-1}D$ is a commutator of $\mathcal U_2$-matrices for all $i\in\{1,...,k\}$. Hence, $B$ is a commutator of $\mathcal U_2$-matrices by Propositions \ref{prop:rmk2.1}-\ref{prop:I_2+prod}. Together with the preceding, $A$ is a product of two commutators of $\mathcal U_2$-matrices.
\end{proof}
%The next result shows that fields with sufficiently many elements reduce the number of commutator factors from Theorem \ref{mainthm} to three.
%    Observe that $4\lfloor\frac{n}{2}\rfloor+5>4\lfloor\frac{n}{2}\rfloor+3>2\lfloor\frac{n}{2}\rfloor+2$ for any $n\in\mathbb N$.
%If $\field$ is infinite, then every matrix in $SL_n(\field)$ is a product of at most three commutators of $\mathcal U_2$-matrices by Proposition \ref{mainlem2} (nonscalar case) and Propositions \ref{mainscalar} and \ref{lem:scalars} (scalar case). 

  We end this paper with the following generalizations of \cite[Theorem 1.1]{hou} and \cite[Theorem 3.5]{wangwu}.
    
    \begin{theorem}\label{mainthm:char2}
        Let $\chara(\field)=2$ and $|\field|\geq2\lfloor\frac{n}{2}\rfloor+2$. Every matrix in $SL_n(\field)$ is a product of at most two commutators of $\mathcal U_2$-matrices and at most four ${\cal U}_2$-matrices.
    \end{theorem}
\begin{proof}
This follows from Proposition \ref{mainlem2} (nonscalar case) and Proposition \ref{mainscalar} (scalar case). With this, the number of ${\cal U}_2$-matrix factors follow from Proposition \ref{rmk:houtowang}.\end{proof}

    \begin{theorem} \label{mainthm6}
        Let $|\mathbb F|\geq4\lfloor\frac{n}{2}\rfloor+5$. Every matrix in $SL_n(\mathbb F)$ is a product of at most three commutators of $\mathcal U_2$-matrices and at most six ${\cal U}_2$-matrices. In addition, if $\mathbb F$ is algebraically closed, then every matrix in $SL_n(\field)$ is a product of at most two commutators of $\mathcal U_2$-matrices and at most four ${\cal U}_2$-matrices.
    \end{theorem}
    \begin{proof}
Since $|\field|\geq4\lfloor\frac{n}{2}\rfloor+5>2n+1$, every matrix in $SL_n(\mathbb F)$ is a product of at most three commutators of ${\cal U}_2$-matrices due to Proposition \ref{mainlem2} (nonscalar case) and Propositions \ref{mainscalar} (i) Propositions \ref{lem:scalars} (scalar case). 

If $\mathbb F$ is algebraically closed (necessarily, $\mathbb F$ is infinite), then every matrix in $SL_n(\field)$ is a product of at most two commutators of $\mathcal U_2$-matrices by Proposition \ref{mainlem2} (nonscalar case) and Propositions \ref{mainscalar} (i) and \ref{lem:scalars} (scalar case). 

The claims about the number of ${\cal U}_2$-matrix factors follow from the calculations above and Proposition \ref{rmk:houtowang}.
    \end{proof}

\section*{Declaration of competing interest}
    
\noindent The authors declare that there is no competing interest.

\section*{Data availability}

\noindent No data was used for the research described in this article.

\section*{Acknowledgements}

\noindent This paper is dedicated to Professor Agnes T. Paras (Institute of Mathematics, UP Diliman) on the occasion of her 60th birthday in 2025. This research is supported by the Computational Research Laboratory Research Grant of the Institute of Mathematics, University of the Philippines Diliman.

%    \subsection{Characteristic 2}
   % In this section, we will see that the characteristic of the field also reduces the number of commutator factors in Theorem \ref{mainthm}

%%\appendix
%%\section{Example Appendix Section}
%%\label{app1}

%%Appendix text.

%%       \citet{<label>} ==> Lamport [21]
%%       \citep{<label>} ==> [21]
%%
%% Example citation, See \citet{lamport94}.




\end{document}